\documentclass[final,leqno,oneeqnum,onefignum,onethmnum,onefignum]{siamltex1213}

\usepackage{amsmath}
\usepackage{amssymb}

\newcommand{\RR}{\mathbb{R}}

\newcommand{\dx}{\textrm{d}x}

\title{On a nonlinear parabolic problem arising \\ in the quantum diffusive description \\
of a degenerate fermion gas
\thanks{This paper has been supported by INDAM (GNFM project ``Quantum fluid-dynamics for systems of identical particles: analytical 
and numerical study''), by the ANR project number ANR-14-ACHN-0030-01 \textit{Kimega} and by the bilateral programme ``Galileo'', project number G14-34.} }

\author{Luigi Barletti\thanks{Dipartimento di Matematica e Informatica ``U.\ Dini'', Universit\`a degli Studi di Firenze, Italy (\email{luigi.barletti@unifi.it}).}
\and 
Francesco Salvarani\thanks{CEREMADE, UMR CNRS 7534, Universit\'e Paris-Dauphine, PSL Research University, 
France \& Dipartimento di Matematica ``F. Casorati'', Universit\`a degli Studi di Pavia, Italy 
(\email{francesco.salvarani@unipv.it}).}
}

\begin{document}

\maketitle
\slugger{siap}{xxxx}{xx}{x}{x--x}

\begin{abstract}
This article studies, both theoretically and numerically, a nonlinear drift-diffusion equation
describing a gas of fermions in the zero-temperature limit.
The equation is considered on a bounded domain whose boundary is divided into an ``insulating'' part, 
where homogeneous Neumann conditions are imposed, and a ``contact'' part, where 
non-homogeneous Dirichlet data are assigned.
The existence of stationary solutions for a suitable class of Dirichlet data is proven by assuming 
a simple domain configuration. 
The long-time behavior of the time-dependent solution, for more complex domain configurations,
is investigated by means of numerical experiments.
\end{abstract}



\pagestyle{myheadings}
\thispagestyle{plain}
\markboth{L.~BARLETTI AND F.~SALVARANI}{A NONLINEAR DIFFUSION PROBLEM FOR A FERMION GAS}

\section{Introduction}
\label{intro}
The derivation of quantum fluid equations from quantum kinetic equations \cite{BarFroMor14, Jungel, TR10} 
is a natural problem in quantum statistical mechanics, in exactly the same way as the derivation of classical fluid equations from Boltzmann equation is a standard topic in classical statistical mechanics 
\cite{bar-gol-lev-89, bar-gol-lev-91, bar-gol-lev-93, bou-gre-pav-sal-13,bou-gre-sal-13,Cerci88,des-mon-sal,gol-stray-04, gol-stray-09}.
It allows, indeed, to clarify the relationships between two levels of descriptions and to obtain ``quantum corrections'' to the classical fluid equations, that are difficult (if not impossible) to identify only on the grounds of physical intuition.
This is particularly true in the case of a quantum system of identical particles, i.e.\ obeying either the Bose-Einstein or the Fermi-Dirac statistics \cite{BaCi,TR10}.
\par
Quantum fluid equations, whose prototype are the Madelung equations \cite{Madelung}, are not only important from a theoretical point of view, but they are very interesting also for applications, in particular in semiconductor devices modelling \cite{BarFroMor14, Jungel, MR1818867, MR1479577}.
Indeed, the fluid description of a quantum system has many practical advantages: first of all, it provides a description in terms of macroscopic variables with a direct physical interpretation
(such as density, current, temperature); moreover, it allows to model open systems in a very natural way, by imposing suitable semiclassical boundary conditions. 
\par
In this article, we study a very specific model, namely the diffusive equation for a degenerate (i.e.\ at zero absolute temperature) gas of fermions in two space dimensions. 
The two-dimensional case is rather peculiar, since the third-order ``quantum pressure'' term, remarkably, vanishes \cite{BaCi, TR10}. The resulting diffusive equation have hence the form of a purely semiclassical equation 
\cite{JKP11}, although being (formally) exact up to order $\hbar^4$.
Such equation (after a suitable rescaling of variables) reads as follows \cite{BaCi,TR10}:
\begin{equation}
\label{PME}
u_t=\nabla\cdot(u \nabla (u+V)),
\end{equation}
where $V$ is a given potential and  $\nabla = ({\partial}/{\partial x_1},{\partial}/{\partial x_2})$, 
and has therefore the form of a two-dimensional ``porous media'' equation  \cite{Vazquez07} endowed with a drift term.
In this paper we choose to work on Equation (\ref{PME}), not only because of its particularly neat form, but also 
because the use of two-dimensional models is natural in many instances of semiconductor devices
\cite{Jungel,MRSbook}.
\par
The mathematical study of Equation (\ref{PME}) has been partially developed in the literature.
In particular, existence and uniqueness of a weak solution in the evolutionary case has been proven in \cite{AL} and in \cite{MR1835610}, whereas an analysis of the long-time behaviour has been provided in \cite{CJMTU} and \cite{MR1866628}.
In \cite{AL}, the time-dependent equation is endowed with mixed homogeneous Neumann and non-homogeneous Dirichlet boundary data; in \cite{CJMTU}, also the stationary equation is considered, but only in the case of homogeneous Neumann boundary data.
\par
However, a semiconductor device cannot be fully described by imposing only homogeneous Dirichlet boundary conditions. Indeed, for describing real situations, a portion of the boundary (corresponding to metallic contacts) should be described by using non-homogeneous Dirichlet data, whereas other regions of the boundary (corresponding to insulating boundaries) should be supplemented with homogeneous Neumann conditions \cite{Jungel,MRSbook}.
\par
The main aim of this article is hence to investigate, both theoretically and numerically, some aspects of Equation (\ref{PME}) by allowing both non-homogeneous Dirichlet boundary data and homogeneous Neumann conditions. 
\par
The structure of the article is the following. 
In the next section we will formulate the mathematical problem and review some basic results on the evolutionary case. 
Then, in section \ref{Sec3}, we will study the stationary case.
Particular attention will be given to the situation in which the domain is rectangular and both the data and the potential depend only on one space variable, in which case solutions that depend only on that variable can be considered.
A theorem of existence and uniqueness for such solutions, requiring some restrictions on the Dirichlet data, will be proven.
Finally, in section \ref{Sec4}, we will describe a numerical procedure for studying the initial-boundary value problem and provide some numerical experiments.

\section{The mathematical problem}
\label{Sec2}

Let $\Omega\subset\RR^2$ be a bounded domain, with a piecewise smooth boundary $\partial \Omega$, and let
$\Gamma_D$ and $\Gamma_N$ be two non-empty subsets of $\partial \Omega$, such that 
$$
  \Gamma_D \cap \Gamma_N =\emptyset 
  \quad \textrm{ and } \quad
  \partial \Omega = \Gamma_D \cup \Gamma_N.
$$
The gas of fermions in the zero-temperature limit is described by a density
function $u: \RR^+\times \Omega\to \RR^+$,
whose time evolution is governed by the following nonlinear drift-diffusion equation \cite{BaCi}:
\begin{equation}
\label{FDequation}
\left\{
\begin{aligned}
&u_t+\nabla\cdot J=0,
\\
&J=-u \nabla \left(u +V\right),
\end{aligned}
\right.
\qquad (t,x)\in \RR^+\times \Omega.
\end{equation}
Here $V : \overline \Omega \to \RR$ is a given potential which, for the sake of simplicity, will be assumed to
be continuously differentiable. 
The problem is supplemented with the initial datum
\begin{equation}
\label{ic}
u(0,x)= u_\mathrm{in}(x),\quad x\in \Omega
\end{equation}
and mixed boundary conditions
\begin{equation}
\label{bc}
\begin{aligned}
&u(t,x)= u_D(x),\quad (t,x)\in \RR^+\times\Gamma_D
\\[6pt]
&J(t,x) \cdot n_x = 0,\quad (t,x)\in \RR^+\times\Gamma_N,
\end{aligned}
\end{equation}
where $u_\mathrm{in}(x) >0$ for a.e.\ $x\in\Omega$, $u_D(x)>0$ for a.e. $x\in \Gamma_D$
and $n_x$ is the outward normal to $\Gamma_N$, at $x\in \Gamma_N$.
\par
Equation $(\ref{FDequation})$,
with non-negative mixed Neumann/Dirichlet boundary conditions
$(\ref{ic})$ and $(\ref{bc})$ has been studied in several works, and 
its main properties have been analyzed.
In particular, existence, uniqueness and preservation of the cone of non-negative functions for
problem \eqref{FDequation}--\eqref{ic}--\eqref{bc} are the specialization to the uncoupled case (without reaction terms) of the 
results obtained in Ref.~\cite{MR1835610}.
More precisely, by adapting to our case the theorems proven in  sections 2 and 3 of Ref.~\cite{MR1835610}, we can state the following:
\begin{theorem}
\label{exist}
Let $u_\mathrm{in}\in L^p(\Omega)$,
$u_D\in L^p(0, T; W^{1,p} (\Omega)) \cap L^\infty((0,T)\times \Omega)$, $p\geq 1$,
and $\partial_t u_D\in L^1(0, T; L^\infty (\Omega))$, with non-negative $u_\mathrm{in}$ and $u_D$.
Then, there exists a unique non-negative weak solution $u\in L^p(0, T; W^{1,p} (\Omega)) +  L^p(0, T; \Xi)$ 
of the initial-boundary value problem $(\ref{FDequation})$--$(\ref{ic})$--$(\ref{bc})$, where
$$
\Xi=\{ \xi \in  W^{1,p} (\Omega) \, : \,  \xi=0 \text{ on }\Gamma_D\}.
$$
\end{theorem}
Another approach for proving this result consists in transforming Equation (\ref{FDequation}) into the equation
\begin{equation}
\label{ALeq}
  \partial_t b(u) =  {\frac{1}{2}} \Delta u + \nabla\cdot\left(b(u) \nabla V \right)
\end{equation}
by means of the transformatiion
\begin{equation}
b(z) := \mathrm{sign}(z) \sqrt{\vert z \vert},
\end{equation}
i.e.
$$
z=u^2\mathrm{sign}(u),
$$
and then by using the existence and uniqueness theory of Alt and Luckhaus, which studied the class of nonlinear parabolic equations written above, with 
boundary conditions of type (\ref{bc}), in \cite{AL}.

In the next section we shall investigate the stationary case and prove (at least for a particular
class of initial/boundary conditions) a theorem of existence and uniqueness of the stationary solution.
Despite of the general result on the evolutionary equation, we are able to prove a sufficient condition which guarantees the well-posedness
of the  stationary problem only for a restricted class of ``supercritical'' Dirichlet data (see Theorem \ref{Theo1D}).
\par
It is worth to remark that the well-posedness of the initial value problem for Equation (\ref{FDequation}) 
with only Neumann conditions ($\Gamma_D = \emptyset$)
has been proven by Carrillo et Al.\ \cite{CJMTU}.
Then, the origin of troubles is clearly in the non-homogeneous Dirichlet conditions.

\section{The stationary problem}
\label{Sec3}
%
The study of non-negative stationary solutions of Equation (\ref{FDequation}), satisfying the prescribed mixed 
boundary conditions (\ref{bc}), leads to consider the problem
\begin{equation}
\label{StationaryFDequation}
\nabla\cdot \left[ u \nabla \left(u +V\right) \right]=0,\quad
x\in \Omega,
\end{equation}
with boundary conditions
\begin{equation}
\label{Stationarybc}
\begin{aligned}
u(x)= u_D(x),\quad x\in \Gamma_D,
\\[6pt]
u \nabla \left( u +V\right) \cdot n_x = 0,\quad x\in \Gamma_N.
\end{aligned}
\end{equation}

\subsection{A general result in the two-dimensional case}
\label{S3.1}
%
A peculiar feature of the stationary problem is the lack of uniqueness of the stationary solution, as shown in the following proposition.
\begin{proposition}
Let us suppose that $u_D \equiv 0$.
If $u \in H^1(\Omega)$ is a non-negative weak solution of the boundary value problem 
$(\ref{StationaryFDequation})$-$(\ref{Stationarybc})$, then, for every $x \in \Omega$,
either $u = 0$ or $\nabla(u+V) = 0$ in $x$. 
\end{proposition}
\begin{proof}
We multiply Equation (\ref{StationaryFDequation}) by  $u+V$ 
and then integrate with respect to $x$ in $\Omega$.
After integrating by parts, we obtain that
$$
\int_\Omega u \left\vert \nabla \left(u +V \right) \right\vert^2 \dx =0,
$$
because of the boundary conditions (\ref{Stationarybc}) with $u_D = 0$. 
Since $u \geq 0$, by assumption, the statement follows. 
\end{proof}
\par
This simple result shows that, depending on $V$, the homogeneous stationary 
problem may have non-unique solution.
For example, if $V\leq 0$ with compact support $K \subset \Omega$, then both $u \equiv 0$
and the function
$$
  u(x) = \left\{ \begin{aligned}
  &0,&  &x \in \overline{\Omega} \setminus K,
  \\
  &-V(x),& &x \in K
 \end{aligned}
   \right.
$$
are non-negative solutions of (\ref{StationaryFDequation}). 
Moreover, if $\Gamma_D = \emptyset$, then $u = \gamma-V$ is a strictly positive solution  
for all constants $\gamma$ such that $\gamma > V(x)$ for all $x\in \overline\Omega$ 
(in particular, if $V$ is allowed to go to $+\infty$, no positive solution may exist at all).
\par

\subsection{Reduction to a one-dimensional problem}
\label{Sec3.2}
%
We now examine the stationary problem in a particular case, which is basically one-dimensional.
Let $\Omega = (0,1)\times(a,b) \subset \RR^2$, let
$$
  \Gamma_D = \{0, 1\}\times[a,b], \qquad \Gamma_N = [0,1] \times\{a,b\},
$$
and let $u = u(x_1,x_2)$, so that the Dirichlet boundary conditions read as follows:
$$
u(0,x_2)= u_0(x_2), \qquad
u(1,x_2)= u_1(x_2),\qquad x_2 \in [a,b], 
$$
and the Neumann conditions are imposed on the two other sides of the rectangle.
\par
Moreover, let us assume that 
$$
  \frac{\partial V}{\partial x_2} = \frac{\partial u_0}{\partial x_2} = \frac{\partial u_1}{\partial x_2} = 0,
$$ 
i.e.\ both the potential $V$ and the Dirichlet data $u_D$ only depend on the variable $x_1$.
Then, clearly, $u(x_1,x_2) = u(x_1)$ is a solution of the stationary problem provided that $u(x_1)$ solves 
the one-dimensional problem
\begin{equation}
\label{1Dstationary}
\begin{aligned}
  &\frac{d}{dx_1}\left[u(x_1)\frac{d}{dx_1}\left(u(x_1) + V(x_1) \right) \right] = 0, \quad x_1 \in [0,1],
  \\[6pt]
  &u(0) = u_0, \quad u(1) = u_1,
\end{aligned}
\end{equation}
where we recall that $V : [0,1] \to \RR$ is assumed to be continuously differentiable. 
We are going to prove that problem (\ref{1Dstationary}) has a unique positive solution for positive
Dirichlet data, $u_0 > 0$ and $u_1>0$, subject to suitable restrictions (see Theorem \ref{Theo1D}).
\par
For the sake of simplicity, from here to the end of this section, we shall denote with $x$ the one-dimensional variable $x_1$; 
then, recall that
$$
   x \equiv x_1.
$$
From (\ref{1Dstationary}) we immediately obtain that a constant $c\in\RR$ exists such that
\begin{equation}
\label{1DprodForm}
  u(x)\left[ u'(x) + V'(x)\right] = c, \qquad x \in [0,1]
\end{equation}
(where the derivatives are now denoted by apices). 
The case $c = 0$ corresponds to all situations in which either $u(x) = 0$ or $u'(x) = -V'(x)$, for some $x\in[0,1]$;
in such case, as discussed above, we cannot expect to have uniqueness.
Then, we assume $c\not=0$ and, consequently, we are forced to consider strictly positive Dirichlet data.
If $c\not=0$, then any regular solution of (\ref{1DprodForm}) cannot vanish in $[0,1]$ and then we can assume
$u$ to be strictly positive in $[0,1]$ and write the differential equation in the normal form
\begin{equation}
\label{1DodeForm}
  u'(x) = \frac{c}{u(x)} - V'(x).
\end{equation}
Let us first consider the case of constant $V'$, for which we have an explicit representation of the solution of the Cauchy
problem for the ODE (\ref{1DodeForm}).
We denote by 
$$
   W : [-1/e,+\infty) \to [-1,+\infty)
   \qquad\text{and}\qquad
   \tilde W : [e,+\infty) \to [1,+\infty), 
$$   
respectively, the inverse functions of $f(x) = x e^x$, and $g(x) =  e^x/x$ ($W$ is known as the 
{\em Lambert  W-function} \cite{NIST}).
It is readily proven that the functions $W$ and $\tilde W$ satisfy the differential equations
\begin{equation}
\label{derW}
   W'(y) = \frac{W(y)}{y\,\big(W(y)+1\big)}
   \qquad\text{and}\qquad
   \tilde W'(y) = \frac{\tilde W(y)}{y\,\big(\tilde W(y)-1\big)}.
\end{equation}
\begin{lemma}
\label{LemmaAlpha}
Let $u_0 >0$, $c>0$ and $\alpha \in \RR$.
Then, the solution of the Cauchy problem
\begin{equation}
\label{1Dalpha}
  u'(x) = \frac{c}{u(x)} - \alpha,
\qquad
  u(x_0) = u_0, 
\end{equation}
is given by $u(x) = \phi(x-x_0,u_0,c, \alpha)$, where
\begin{equation}
\label{phiDef}
 \phi(\Delta x,u_0,c, \alpha) =
\left\{ \begin{aligned}
  &\frac{c}{\alpha}\left\{1+W\left[\left( \frac{\alpha u_0}{c} - 1\right)
  e^{-\frac{\alpha^2}{c}\Delta x + \frac{\alpha u_0}{c} - 1} \right]
  \right\},&  &\text{if $\alpha >0$} ,
  \\[6pt]
  &\sqrt{u_0^2 + 2c\,\Delta x},& &\text{if $\alpha = 0$,}
  \\[6pt]
  &\frac{c}{\alpha}\left\{1- \tilde W\left[\left(1- \frac{\alpha u_0}{c}\right)^{-1}
  e^{\frac{\alpha^2}{c}\Delta x + 1- \frac{\alpha u_0}{c}}
   \right]\right\},& &\text{if $\alpha <0$.}
 \end{aligned}
   \right.
\end{equation}
Moreover, for fixed $x > x_0$, we have that $\phi(x,u_0,c, \alpha)$ is strictly increasing with respect to $u_0$ and $c$, 
and strictly decreasing with respect to $\alpha$.
\end{lemma}
\begin{proof}
The fact that $\phi$ satisfies \eqref{1Dalpha} comes straightforwardly from \eqref{derW}.
Then we just prove the growth properties with respect to $\alpha$ and $c$.
Although they could be checked directly on expression \eqref{phiDef}, it is easier to 
use the differential equation satisfied by $\phi$, i.e.
\begin{equation}
\label{auxnew1}
  \frac{\partial \phi}{\partial x} = \frac{c}{\phi} - \alpha.
\end{equation}
Defining $\psi = {\partial \phi}/{\partial c}$ we obtain
\begin{equation}
\label{auxnew2}
   \frac{\partial \psi}{\partial x} = -\frac{c\,\psi}{\phi^2} + \frac{1}{\phi}
\end{equation}
and, since $\phi_{|x=0} = u_0$ for all $\alpha \in \RR$, we also have 
\begin{equation}
\label{auxnew3}
\psi_{|x=0} = 0.
\end{equation}
If $x_0\in(0,1]$ existed such that $\psi \leq 0$ in $(0,x_0)$, then from \eqref{auxnew2} one would obtain
$$
  \frac{\partial \psi}{\partial x} \geq \frac{1}{\phi} > 0.
$$
in $(0,x_0)$ (where the fact that $\phi > 0$ can be easily checked from \eqref{phiDef}) which, together with \eqref{auxnew3},
would give $\psi > 0$ in $(0,x_0)$.
Then $\psi > 0$ in $(0,1]$, which of course implies that $\phi$ is strictly increasing with respect to $c>0$.
\par \noindent
The proof that $\phi$ is strictly decreasing with respect to $\alpha \in \RR$ can be carried out analogously.
\end{proof}
\par \smallskip\noindent
In addition to the properties listed in Lemma \ref{LemmaAlpha} let us also note 
that $\phi(\Delta x,u_0,c,\alpha)$, as $\Delta x$ increases from $0$ to $+\infty$:
\begin{enumerate}
\item[(i)]
increases monotonically from $u_0$ to $+\infty$, if $\alpha \leq 0$;
\item[(ii)]
increases monotonically from $u_0$ to the asymptotic value $c/\alpha$, if $\alpha > 0$ and  $u_0 < c/\alpha$;
\item[(iii)]
is identically equal to $c/\alpha$, if $\alpha > 0$ and  $u_0 = c/\alpha$;
\item[(iv)]
decreases monotonically from $u_0$ to the asymptotic value $c/\alpha$, if $\alpha > 0$ and  $u_0 > c/\alpha$.
\end{enumerate}
Lemma \ref{LemmaAlpha} will allow us to prove properties of the more general differential equation (\ref{1DodeForm})
by approximating $V'(x)$ with piecewise constant functions.
\begin{lemma}
\label{LemmaPW}
Let $x_0 = 0 < x_1 < \cdots < x_n = 1$ be a partition of $[0,1]$ and let $V'_n(x)$ be a piecewise constant
function, taking the value $\alpha_i$ in the interval $(x_i, x_{i+1})$, such that $V'_n \to V'$ uniformly in $[0,1]$,
as $n\to +\infty$.
Let $u(x)$ be the solution of the Cauchy problem
\begin{equation}
\label{1DCP}
  u'(x) = \frac{c}{u(x)} - V'(x), \quad x \in [0,1], \qquad u(0) = u_0 >0,
\end{equation}
with $c>0$, and let $v_n(x)$ be defined by
$$
\begin{aligned}
  &v_n(x) = \phi(x,u_0,c,\alpha_0),& \ &x \in [0,x_1],
\\[4pt]
  &v_n(x) = \phi(x-x_i,v_n(x_i),c,\alpha_i),& \ &x \in (x_i,x_{i+1}],
  \quad i = 1,2,\ldots,n-1
  \end{aligned}
$$
(in other words, $v_n$ solves $v_n'(x) = \frac{c}{v_n(x)} - \alpha_i$  in $[x_i,x_{i+1}]$,
taking as initial value in each interval the final value of the preceding interval, starting with $u_0$ in the first interval).
Then $v_n \to u$ uniformly in $[0,1]$, as $n\to +\infty$. 
\end{lemma}
\begin{proof}
Note that $u$ and $v_n$ are continuous solutions of the integral equations,
$$
  u(x) = u_0 + \int_0^x \frac{c}{u(y)}\,dy - V(x)+V(0)
$$
and
$$
  v_n(x) = u_0 + \int_0^x \frac{c}{v_n(y)}\,dy - V_n(x)+V(0),
$$
respectively (where, of course, $V_n(x) := V(0) + \int_0^x V_n'(y)\,dy$ is a piecewise linear approximation of $V(x)$).
Then,
$$
  \vert u(x)-v_n(x)\vert \leq  c\int_0^x \frac{ \vert u(y)-v_n(y)\vert}{\vert u(y)\, v_n(y) \vert}\,dy + \vert V(x) - V_n(x) \vert.
$$
We note that $u(x)$ can be decreasing only in the set  
$$
  \{x \in [0,1] \mid V'(x) > 0\ \text{and}\ u(x) > c/V'(x)\}.
$$
Since  $u(0) = u_0 >0$, we have that
$$
  u(x) > \min\{u_0, c/V'_+ \}, 
$$ 
where $V'_+$ denotes the maximum positive value of $V'(x)$ in $[0,1]$.
Similar considerations hold for $v_n$ (with a lower bound that can be supposed to be independent on $n$).
Then, $\epsilon >0$ exists such that $u(x)v_n(x) \geq \epsilon$ for all $x\in[0,1]$ and, therefore,
$$
  \vert u(x)-v_n(x)\vert \leq  \frac{c}{\epsilon} \int_0^x \vert u(y)-v_n(y)\vert\,dy + \vert V(x) - V_n(x) \vert.
$$
Since $V_n(x) \to V(x)$ uniformly in $[0,1]$, the thesis follows from Gronwall's lemma.
\end{proof}
\begin{lemma}
\label{LemmaBounds}
Let $u$ as in the previous lemma and let
$$
  \alpha_M = \max_{0\leq x \leq 1} V'(x),
  \qquad
  \alpha_m = \min_{0\leq x \leq 1} V'(x).
$$
Then,
\begin{equation}
\label{ubound}
  \underline v(x) \leq u(x) \leq \overline v(x), \qquad x \in [0,1],
\end{equation}
where
$$
  \underline v(x) = \phi(x,u_0,c,\alpha_M),
  \qquad
   \overline v(x) = \phi(x,u_0,c,\alpha_m).
$$
\end{lemma}
\begin{proof}
Let $V'_n$ and $v_n$ be as in the previous lemma, and assume, without loss of generality, that 
$\alpha_m \leq V'_n(x) \leq \alpha_M$ for all $x\in[0,1]$.
In the first interval, $[0,x_1]$, we have $\underline v(x) \leq v_n(x) \leq \overline v(x)$,
because 
$$
  \phi(x,u_0,c,\alpha_M) \leq \phi(x,u_0,c,\alpha_0) \leq \phi(x,u_0,c,\alpha_m),
$$
as $\phi$ is increasing for decreasing $\alpha$ (Lemma \ref{LemmaAlpha}).
Then, in each of the successive intervals we still have $\underline v(x) \leq v_n(x) \leq \overline v(x)$,
because, {\em a fortiori},
$$
  \phi(x,\underline v(x_i),c,\alpha_M) \leq \phi(x,v_n(x_i),c,\alpha_0) \leq \phi(x, \overline v(x_i),c,\alpha_m),
$$
being $\phi$ also increasing with respect to the initial value.
Thus $\underline v(x) \leq v_n(x) \leq \overline v(x)$ in the whole interval $[0,1]$. 
Since, from Lemma \ref{LemmaPW}, $v_n \to u$ uniformly in $[0,1]$ the inequalities are also true for $u$,
which proves our claim.
\end{proof}
Note that the result of Lemma \ref{LemmaBounds} implies, in particular, that $u(x) > 0$ for all $x \in [0,1]$.
\begin{lemma}
\label{LemmaMonotonia}
Let us denote by $u(x,c)$, with $x\in [0,1]$ and $c>0$, the solution of the Cauchy problem $(\ref{1DCP})$.
Then, $u(x,c)$ is strictly increasing with respect to $c$, for every fixed $x \in (0,1]$.
\end{lemma}
\begin{proof}
We resort again to the uniformly approximating sequence defined in Lemma \ref{LemmaPW}, that we now denote 
by $v_n(x,c)$ in order to stress the dependence on $c$. 
We can assume, as in Lemma \ref{LemmaBounds}, that $\alpha_m \leq V'_n(x) \leq \alpha_M$ (independently on $n$).
\par \noindent
From definition \eqref{phiDef} it is apparent that  $\phi(\Delta x, u_0,c,\alpha)$ is a $C^1$-function of $c > 0$, piecewise continuous
with respect to $\alpha \in \RR$. 
However, if we prove that $\frac{\partial\phi}{\partial c}$ is actually continuous with respect to $\alpha$ also at $\alpha = 0$, then
$\frac{\partial\phi}{\partial c}$ has a (positive) lower bound when $\alpha$ varies in $[\alpha_m,\alpha_M]$ and, therefore,
for any given $\Delta x \in (0,1]$,  $u_0>0$ and $c>0$, a constant $\mu > 0$ exist such that
\begin{equation}
\label{muaux}
  \frac{\partial}{\partial c} \phi(\Delta x, u_0,c,\alpha_i)  \geq \mu, 
  \quad 
  \text{for all $i = 0,1,\ldots,n-1$.}
\end{equation}
We now prove the continuity of $\frac{\partial\phi}{\partial c}$ (and of $\phi$ as well) as $\alpha \to 0^+$.
For $\alpha > 0$ the expression of $\phi$ (which is now seen as a function of $c$ and $\alpha$) is the first one in \eqref{phiDef}, i.e.
$$
  \phi(c,\alpha) = \frac{c}{\alpha} \left\{ 1 + W[g(c,\alpha)] \right\},
$$
where 
$$
  g(c,\alpha) = \left( \frac{\alpha u_0}{c} - 1\right) e^{-\frac{\alpha^2}{c}\Delta x + \frac{\alpha u_0}{c} - 1}.
$$
By using \eqref{derW} we easily get
$$
 \frac{\partial}{\partial \alpha} \{1 + W[g(c,\alpha)]\}^2 = \frac{2\alpha W[g(c,\alpha)]  (u_0^2 + 2c\Delta x - 2\alpha u_0 \Delta x)}%
 {\alpha c u_0 - c^2}. 
$$
Since $W[g(c,\alpha)] \to -1$ as $\alpha \to 0^+$, by applying de l'H\^opital's theorem we obtain
$$  
\lim_{\alpha \to 0^+} \frac{c^2}{\alpha^2} \{1 + W[g(c,\alpha)]\}^2 = u_0^2 + 2c\Delta x
$$
so that
\begin{equation}
\label{muaux2}
  \lim_{\alpha \to 0^+} \phi(c,\alpha) =  \lim_{\alpha \to 0^+}\frac{c}{\alpha} \{1 + W[g(c,\alpha)]\} = 
  \sqrt{u_0^2 + 2c\Delta x},
\end{equation}
which, according to \eqref{phiDef}, proves the continuity of $\phi$ for $\alpha \to 0^+$.
\par \noindent
Coming to the $c$-derivative, after straightforward calculations, we have (for $\alpha > 0$)
$$
   \frac{\partial \phi}{\partial c} (c,\alpha) = \frac{1}{\alpha} \{1 + W[g(c,\alpha)]\}
   + \frac{\alpha W[g(c,\alpha)](\alpha u_0 \Delta x - c\Delta x - u_0^2)}{c \{1 + W[g(c,\alpha)]\} (\alpha u_0 - c)}.
$$
By using \eqref{muaux2} we obtain therefore
$$
 \lim_{\alpha \to 0^+}  \frac{\partial \phi}{\partial \alpha} (c,\alpha) = \frac{\Delta x}{\sqrt{u_0^2 + 2c\Delta x}}
  = \frac{\partial}{\partial c} \sqrt{u_0^2 + 2c\Delta x},
$$
which proves the continuity of $\frac{\partial\phi}{\partial c}$ as $\alpha \to 0^+$.
\par \noindent
The continuity of $\frac{\partial\phi}{\partial c}$ (and $\phi$) for $\alpha \to 0^-$ is proven in the same way by using the expression 
of $\phi$ for $\alpha < 0$ in \eqref{phiDef}.
This, according to the discussion above, proves  \eqref{muaux}.
We stress the fact that  $\mu$ only depends on $\alpha_m$ and $\alpha_M$, and does not depend on the 
sequence of the $\alpha_i$'s.
Then, from the definition of $v_n$ (see Lemma \ref{LemmaPW}) we have that 
$$
  v_n(x,c+\Delta c) - v_n(x,c) \geq  \mu \Delta c,
$$
for all $x \in (0,1]$ and $\Delta c \geq 0$ (small enough), with $\mu$ independent on $n$.
Then, passing to the limit for $n\to\infty$ we obtain that $u(x,c)$ is strictly increasing with respect to 
$c$ for every $0<x\leq 1$.
\end{proof}
\begin{lemma}
\label{LemmaCauchy}
Referring to Lemma \ref{LemmaMonotonia} for the notations, we have that $u(x,c)$ converges uniformly 
to a continuous limit $u(x,0)$ as $c\to 0^+$.
\end{lemma}
\begin{proof}
It is not difficult to show that
\begin{equation}
\label{phiC0}
\lim_{c\to 0^+} \phi(\Delta x,u_0,c,\alpha) = \max \left\{u_0 - \alpha x, 0 \right\},
\end{equation}
uniformly with respect to $\Delta x$, $u_0$ and $\alpha$. 
Then,  still assuming $\alpha_m \leq V'_n(x) \leq \alpha_M$, we obtain that $\lim_{c\to 0^+}v_n(x,c) = v_n(x,0)$, 
uniformly with respect to $x$ and $n$, where  $v_n(x,0)$ is a continuous, piecewise linear, limit function.
Since $c \mapsto v_n(x,c)$ (seen as a sequence in $c$) is uniformly Cauchy with respect to $x$ and $n$, 
it is straightforward to prove that  $c \mapsto u(x,c)$ is in turn uniformly Cauchy (with respect to $x$) and, 
therefore, converges uniformly to a continuous limit $u(x,0)$.
\end{proof}
\begin{definition}
\label{DefiU}
To any given data $(u_0,V)$ of the Cauchy problem $(\ref{1DCP})$ we associate the set of points
$$
  1 < x_1 \leq y_1  < x_2  \leq y_2 < \cdots < x_n = y_n  =1 
$$
by means of the following recursive rule (where $V_0 = V(0)$):
\begin{enumerate}
\item[1.]
we start by putting
$$
  x_1 = \sup \big\{ x \in [0,1] \ \big\vert \ u_0 + V_0 - V(\xi) > 0 \ \forall\ \xi \in [0,x)  \big\};
$$
if $x_1 = 1$, then we put $y_1 = 1$ and the procedure ends with $n=1$, otherwise we proceed to step 2.
\item[2.]
If $x_1 < 1$ we put
$$
  y_1 = \sup \big\{ x \in [x_1,1] \ \big\vert \ -V'(\xi) \leq 0\ \forall\ \xi \in [x_1,x) \big\};
$$
if $y_1 = 1$, then the procedure ends with $n=1$, otherwise we proceed to step 3.
\item[3.]
If $y_{i-1} < 1$ we put
$$
  x_i = \sup \big\{ x \in (y_{i-1},1] \ \big\vert \ V(y_{i-1}) - V(\xi) > 0 \ \forall\ \xi \in (y_{i-1},x)  \big\};
$$
if $x_i = 1$, then we put $y_i = 1$ and the procedure ends with $n=i$, otherwise we proceed to step 4.
\item[4.]
If $x_i < 1$ we put
$$
  y_i = \sup \big\{ x \in [x_i,1] \ \big\vert \ -V'(\xi) \leq 0\ \forall\ \xi \in [x_i,x) \big\};
$$
if $y_i = 1$, then the procedure ends with $n=i$, otherwise we increment the index $i$ and repeat steps 3 and 4
until we find $x_i = 1$ or $y_i = 1$ for some $i$.
(For the sake of simplicity we can assume that $V$ changes sign a finite number of times,
so that the procedure stops at a finite $n$.)
\end{enumerate}
Finally, we define the continuous function
\begin{equation}
\label{Udef}
U(x) = \left\{
\begin{aligned}
 &u_0 + V_0 - V(x),& &x\in [0,x_1]
\\
&0,& &x \in [x_i,y_i],
\\
&V(y_i) - V(x),& &x\in [y_i,x_{i+1}].&  
\end{aligned}
\right.
\end{equation}
\end{definition}
\begin{lemma}[Asymptotic behaviour for $c\to 0^+$]
\label{Lemmac0}
Let $u$ be the solution of the Cauchy problem $(\ref{1DCP})$. 
Then, 
$$
  \lim_{c\to 0^+} u(x) = U(x),
$$ 
where $U$ is the function defined in Definition  \ref{DefiU}.
\end{lemma}
\begin{proof}
We divide the proof into three recursive steps.
\par \noindent
{\em Step 1.} 
Let $\epsilon>$ be arbitrarily small and let $x \in [0,x_1-\epsilon]$.
From Definition \ref{DefiU} we have that $\eta>0$ exists such that $u_0 + V_0 - V(\xi) \geq \eta$ for all $\xi \in[0,x]$.
Then, from the obvious inequality
$$
u(x) =  u_0 +V_0 - V(x) + \int_0^x \frac{c}{u(\xi)}\,d\xi  > u_0 +V_0 - V(x),
$$
we obtain
$$
  u(x) -  u_0 - V_0 + V(x) \leq \frac{cx}{\eta}, \qquad x \in [0,x_1-\epsilon],
$$
which proves that $\lim_{c\to 0^+} u(x) = U(x)$, for all $x \in [0,x_1)$. 
Since the limit has to be a continuous function (Lemma \ref{LemmaCauchy}), from the arbitrariness 
of $\epsilon$ we also have
$$
  \lim_{c\to 0^+} u(x_1) = U(x_1) = 0.
$$
{\em Step 2.}
In the interval $[x_1,y_1]$ we have that $\lim_{c\to 0^+} u(x_1) = 0$ (from the previous step)
and, by definition, $\alpha_m = \min_{x\in[x_1,y_1]} V'(x) \geq 0$. 
Then, from Lemma \ref{LemmaBounds} 
$$
  \lim_{c\to 0^+}u(x) \leq  \lim_{c\to 0^+} \sqrt{u(x_1) + 2c(x-x_1)} =0,
$$
for all $x\in [x_1,y_1]$.
\par \noindent
{\em Step 3.}
In the subsequent intervals of the form $[y_{i},x_{i+1}]$ or $[x_i,y_i]$, we repeat the same proofs of the steps 1 and 2
(respectively), with the only difference that $U(y_i) = 0$ instead of $U(0) = u_0 > 0$ (which implies that Step 1 has to be
slightly modified by using a $\eta >0$ such that $V(y_i) - V(\xi) \geq \eta$ for all $\xi \in [y_i+\epsilon, x_{i+1}-\epsilon]$).
 \end{proof}
\par
Note that the limiting process $c\to 0^+$ selects one particular solution among the infinitely many 
solutions of the case $c=0$, for which we have seen that there is not uniqueness (see section \ref{S3.1}). 
\begin{definition}
\label{DefiCritical}
We define the critical values for the Dirichlet data $u_0>0$ and $u_1 >0$ as follows:
\begin{equation}
  u_0^\mathrm{crit} = V_M - V_0, 
 \qquad\quad
  u_1^\mathrm{crit} =  V_M - V_1, 
\end{equation}
where we put $V_0 = V(0)$, $V_1 = V(1)$ and
$$
  V_M =  \max_{x\in[0,1]} V(x).
$$
If a boundary term is greater than the corresponding critical value, then it is said to be {\it supercritical}, otherwise it is said to be {\it subcritical}.
\end{definition}
\par
We can finally prove that the solution of  the  the Dirichlet problem (\ref{1Dstationary}) exists and is unique,
provided that the Dirichlet data $u_0 >0$  and $u_1 >0$ are not both subcritical.
\begin{theorem}
\label{Theo1D}
If $u_0 >  u_0^\mathrm{crit}$, then the Dirichlet problem $(\ref{1Dstationary})$ has a unique, strictly positive, 
solution $u$ for all $u_1 >0$.
\\
If $u_1 >  u_1^\mathrm{crit}$, then the Dirichlet problem $(\ref{1Dstationary})$ has a unique, strictly positive, 
solution $u$ for all $u_0 >0$.
\end{theorem}
\begin{proof}
Consider first the case $u_0 >  u_0^\mathrm{crit}$ and assume, temporarily, that $c>0$.
\par
We know that the mapping $c \mapsto u(1)$, obtained by solving the Cauchy problem (\ref{1DCP}) is 
strictly monotone. Moreover, from Lemma \ref{LemmaBounds}, we have that 
$$
  \lim_{c\to+\infty} u(1) \geq  \lim_{c\to+\infty} \underline v(1)  = +\infty, 
 $$
since it is easy to see that $\phi(x,u_0,c,\alpha)\to+\infty$ as $c \to +\infty$, for all $x>0$ and $\alpha \in \RR$.
Moreover, since $u_0 > u_0^\mathrm{crit}$, from Definition \ref{DefiU} and Lemma \ref{Lemmac0} we have that
$$
   \lim_{c\to 0^+} u(x) = U(x) = u_0 + V_0 - V(x) > 0,
$$
i.e., in particular,
$$
   \lim_{c\to 0^+} u(1) = u_0 - \Delta V > 0, \quad \text{where \quad $\Delta V = V_1 - V_0$}.
$$
What we have proven so far is that the left-Dirichlet datum $u_0 >  u_0^\mathrm{crit}$ can be uniquely linked to 
any right-Dirichlet datum $u_1 = u(1) \in [u_0 - \Delta V, +\infty)$, by solving the Cauchy problem (\ref{1DCP}) with 
a suitable $c \geq 0$.
To go below the threshold  $u_0 - \Delta V$ we have to consider negative values of $c$.  
However, the case $c < 0$ can always be recast into the case $c >0$ by noticing that, if $u$ satisfies
$u'(x) = c/u(x) - V'(x)$, then $\tilde u (x) = u(1-x)$ satisfies 
$$
  \tilde u'(x) = -\frac{c}{\tilde u(x)} - \tilde V '(x), \quad \text{where \quad $\tilde V(x) = V(1-x)$}.
$$
In other words, the forward Cauchy problem with $c<0$ is equivalent to the backward Cauchy problem with $c>0$.
Hence, if we fix the right-Dirichlet datum $0 < u_1 < u_0 - \Delta V$ and take $c<0$, we can consider the 
Cauchy problem
\begin{equation}
\label{1DCPbw}
  \tilde u'(x) = \frac{\vert c \vert}{\tilde u(x)} - \tilde V'(x), \quad x \in [0,1], \qquad \tilde u(0) = u_1,
\end{equation}
that is the backward Cauchy problem for $u$ with right-Cauchy datum $u_1$. 
What we want to prove is that $u_1$ can be linked to $u_0$ with a suitable choice of $c<0$. i.e., that $c<0$ exists
such that  $ \tilde u(1) = u(0) = u_0$.
Since the datum $u_1$ is not necessarily supercritical for problem (\ref{1DCPbw}), then in the limit $c\to 0^-$ 
we have in general that $\tilde u$ tends to the asymptotic solution $\tilde U (x)$ (with the obvious definition) and, 
according to the theory we already know, for $\tilde u (1)$ there are three possibilities:
$$
  \lim_{c\to 0^-} \tilde u(1) = 
  \left\{
  \begin{aligned}
  &u_1 - \Delta \tilde V,&  &\text{(that is $u_1 + \Delta V$),}
  \\
  &0,
  \\
  &V (y) - V(0),& &\text{for some point $y \in (0,1)$.}
  \end{aligned}
  \right.
$$
In the first case (corresponding to $u_1$ supercritical for problem (\ref{1DCPbw}) we have 
$ \lim_{c\to 0^-}  \tilde u(1) = u_1 + \Delta V < u_0$ (since we have assumed  $u_1 < u_0 - \Delta V$).
In the second case, obviously, $ \lim_{c\to 0^-}  \tilde u(1) = 0 < u_0$ and in the third case 
$$
  u_0 > -V_0 + V(y) =  \lim_{c\to 0^-}  \tilde u(1),
$$
(because $u_0$ is supercritical).
In each of the three cases, therefore, 
$$
  \lim_{c\to 0^-}  u(0) =  \lim_{c\to 0^-}  \tilde u(1) < u_0
$$
and then (owing to the monotonic growth with respect to $\vert c \vert$), $u_0 = u(0)$ for a suitable $c<0$.
In conclusion, we have shown that the left-Dirichlet datum $u_0 >  u_0^\mathrm{crit}$ can be uniquely linked to 
any right-Dirichlet datum $u_1 = u(1)>0$, by solving the Cauchy problem (\ref{1DCP}) with 
a suitable $c \in \RR$. 
\par
The proof of the second part of the claim, i.e.\ the case $u_1 >  u_1^\mathrm{crit}$, is completely equivalent
to the proof of the first part since, as we have just shown, it suffices to change $c$ into $-c$.
\end{proof}
%
\section{Numerical simulations}
\label{Sec4}
%
Let $x=(x_1,x_2)$ and let $\Omega$ the square in $\RR^2$ defined as
$$
\Omega=\{ x\in \RR^2 \, : \, 0\leq x_1 \leq 1 \textrm{ and } 0\leq x_2 \leq 1 \}.
$$
We denote moreover 
$$
\partial\Omega:=\Gamma=\sum_{i=1}^5\Gamma_i,
$$
where
$$
\begin{array}{ll}
\Gamma_1&:=\{ x\in \RR^2 \, : \, 0 < x_1 \leq 1 \textrm{ and } x_2=0 \}; \\
\Gamma_2&:=\{ x\in \RR^2 \, : \, x_1 = 1 \textrm{ and } 0 < x_2 \leq 1 \}; \\
\Gamma_3&:=\{ x\in \RR^2 \, : \, 0 \leq x_1 < 1 \textrm{ and } x_2=1 \}; \\ 
\Gamma_4&:=\{ x\in \RR^2 \, : \, x_1 = 0 \textrm{ and } 1/2\leq x_2 < 1 \};  \\
\Gamma_5&:=\{ x\in \RR^2 \, : \, x_1 = 0 \textrm{ and } 0\leq x_2 < 1/2 \} \\
\end{array}
$$
and consider the unknowns $u\, : \, \RR^+\times\Omega
\to \RR$ and $J\, : \, \RR^+\times\Omega
\to \RR^2$.

We now introduce, for $n=0,\dots, N$, a semi-discrete weak formulation of Equation (\ref{FDequation}) for the
semi-discrete density unknowns $u^n(x)=u(n\Delta t, x)$ and the semi-discrete flux $J^n(x)=J(n\Delta t, x)$, 
where $\Delta t>0$ is the time step.
\par
Let $\phi \in H^1(\Omega)$ and $\psi\in H^1(\Omega)\times H^1(\Omega)$. 
We approximate the weak formulation of Equation (\ref{FDequation}) by means
of the following coupled system
\begin{equation}
\left\{
\label{weak}
\begin{aligned}
 &\int_\Omega \frac{1}{\Delta t} (u^n-\max\{ u^{n-1}, 0\})\phi\, dx 
+\int_{\partial\Omega} \phi J^n\cdot n_x\, dS
-\int_\Omega J^n\cdot\nabla \phi\, dx =0,
\\[6pt]
 &\int_\Omega J^n\cdot\psi\, dx + \frac{1}{2\pi}\int_\Omega
\max\{ u^{n-1}, 0\} [\nabla u^n\cdot \psi]\, dx
+ \int_\Omega u^{n}[\nabla V\cdot \psi]\, dx = 0,
\end{aligned}
\right. 
\end{equation}
where $n_x$ is the outward normal with respect to $\Omega$ starting from a point
$x\in\partial\Omega$.
\par
The boundary term (i.e.\ the integral on $\partial\Omega$ in the formulation written above) 
will be treated in agreement with the different boundary conditions specified for each numerical simulation.
\par
The ``positive part'' term in the weak formulation (\ref{weak}) is pleonastic at the continuous level, 
since the solution of the problem is known to be non negative by Theorem \ref{exist}.
At the discrete level, this strategy helps in controlling the non-negativity of the numerical solution.
\par
The numerical experiments aim at showing some peculiar properties
of Equation (\ref{FDequation}) under different choices of the potential $V$
and of the boundary conditions, and have been obtained by using the Finite Element Method.
\par
From the semi-discrete formulation, by using quadratic $P_2$ Lagrangian elements on a triangular mesh, 
we obtain a linear system whose size is given by twice the number of vertices and the number of mid-edges
in the triangulation. The system is solved by a multi-frontal Gauss LU factorization.
\par
The simulations are written in \texttt{FreeFem++}. 
The mesh discretization used in our simulations is composed by 19514 triangles, with 9940 vertices. 
\subsection{Long-time behavior of the solution in the one-dimensional case}
\label{Subsec:0}
We consider here the boundary conditions
\begin{equation}
\label{1Dcond}
J|_{\Gamma_1}=J|_{\Gamma_3}=0,\quad u|_{\Gamma_2}=u_1,\qquad 
u|_{\Gamma_4\cup \Gamma_5}=u_0.
\end{equation}
and assume that both the initial datum and the potential only depend on $x_1$. 
Then, as already discussed in section \ref{Sec3.2} in the stationary case,  the two-dimensional problem reduces to the one-dimensional 
problem
\begin{equation}
\label{1Devolutive}
\begin{aligned}
  &\frac{\partial u}{\partial t} =  \frac{\partial}{\partial x_1}\left[u \frac{\partial}{\partial x_1}\left(u + V \right) \right] = 0, 
  \quad x_1 \in [0,1],\ t > 0,
  \\[6pt]
  &u(0,t) = u_0, \quad u(1,t) = u_1,\quad u(x_1,0) = u_\mathrm{in}(x_1),
\end{aligned}
\end{equation}
where the initial datum, $u_\mathrm{in}(x_1)$, and the (constant) Dirichlet data, $u_0$ and $u_1$, are positive.
In this subsection we report a set of numerical simulations showing that the solution to Equation (\ref{1Devolutive}) 
tends asymptotically to the stationary solution discussed in section \ref{Sec3.2} (see, in particular, Theorem \ref{Theo1D}).
In Figure \ref{fig:0} the spatial profile of the solution along the direction $x_1$ (recall that the solution is homogeneous in
the direction $x_2$) is shown at different instants of time.
This first set of simulations has been performed with the potential
\begin{equation}
\label{Vsin}
  V(x_1) =  \sin(2\pi x_1).
\end{equation}
Note that, according to Definition \ref{DefiCritical}, for such potential we have 
$$
u_0^\mathrm{crit} = u_1^\mathrm{crit} = 1.
$$
The initial datum $u_\mathrm{in}(x_1)$ is chosen as a linear function interpolating the values $u_0$ and $u_1$
and the evolution of such datum towards the asymptotic, stationary solution (dotted black curve) is illustrated for different
choices of $u_0$ and $u_1$.
\begin{figure}[h!]
\begin{center}
(a)\includegraphics[height=4.2cm]{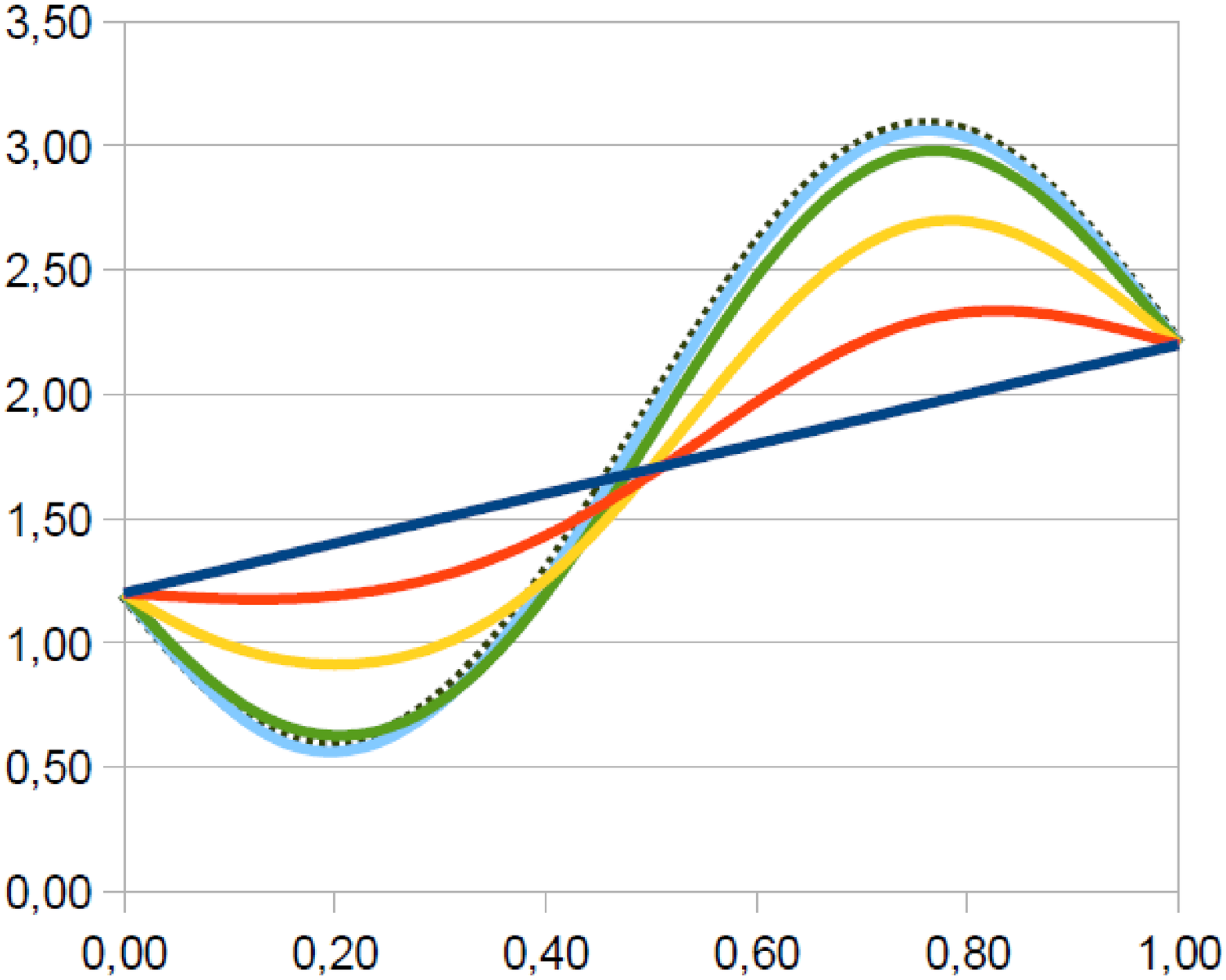}\
(b) \includegraphics[height=4.2cm]{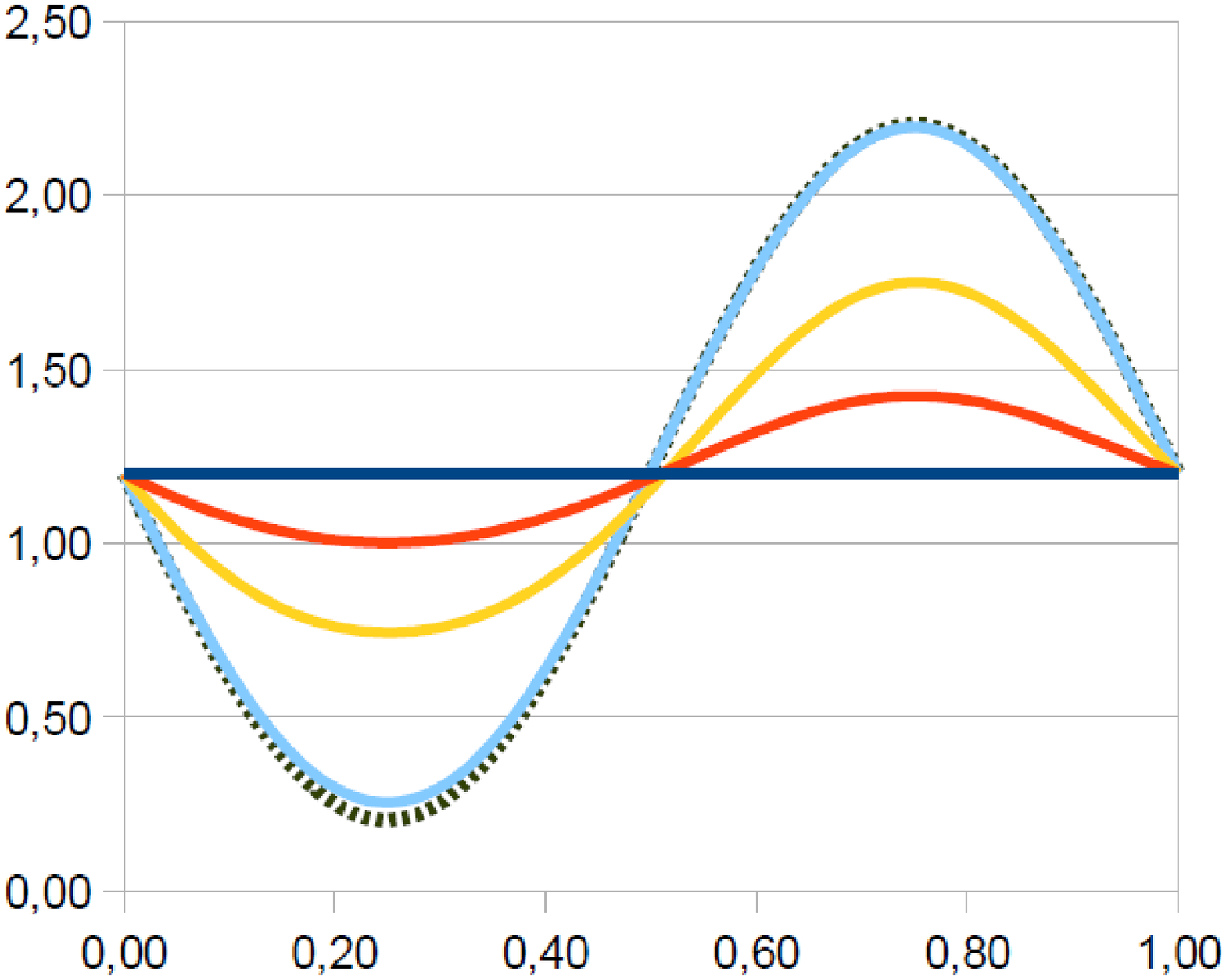}\\
(c) \includegraphics[height=4.2cm]{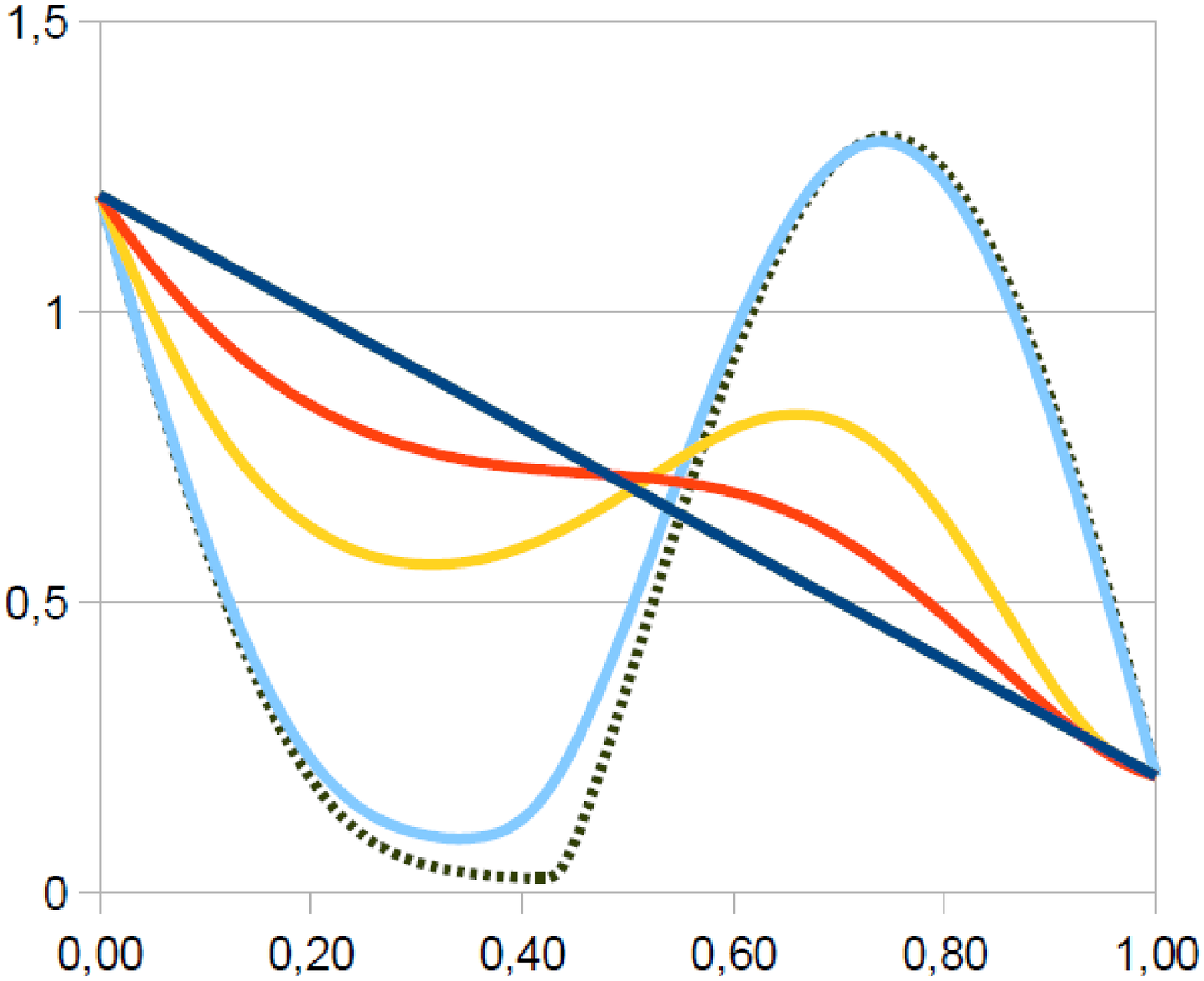}\
(d) \includegraphics[height=4.2cm]{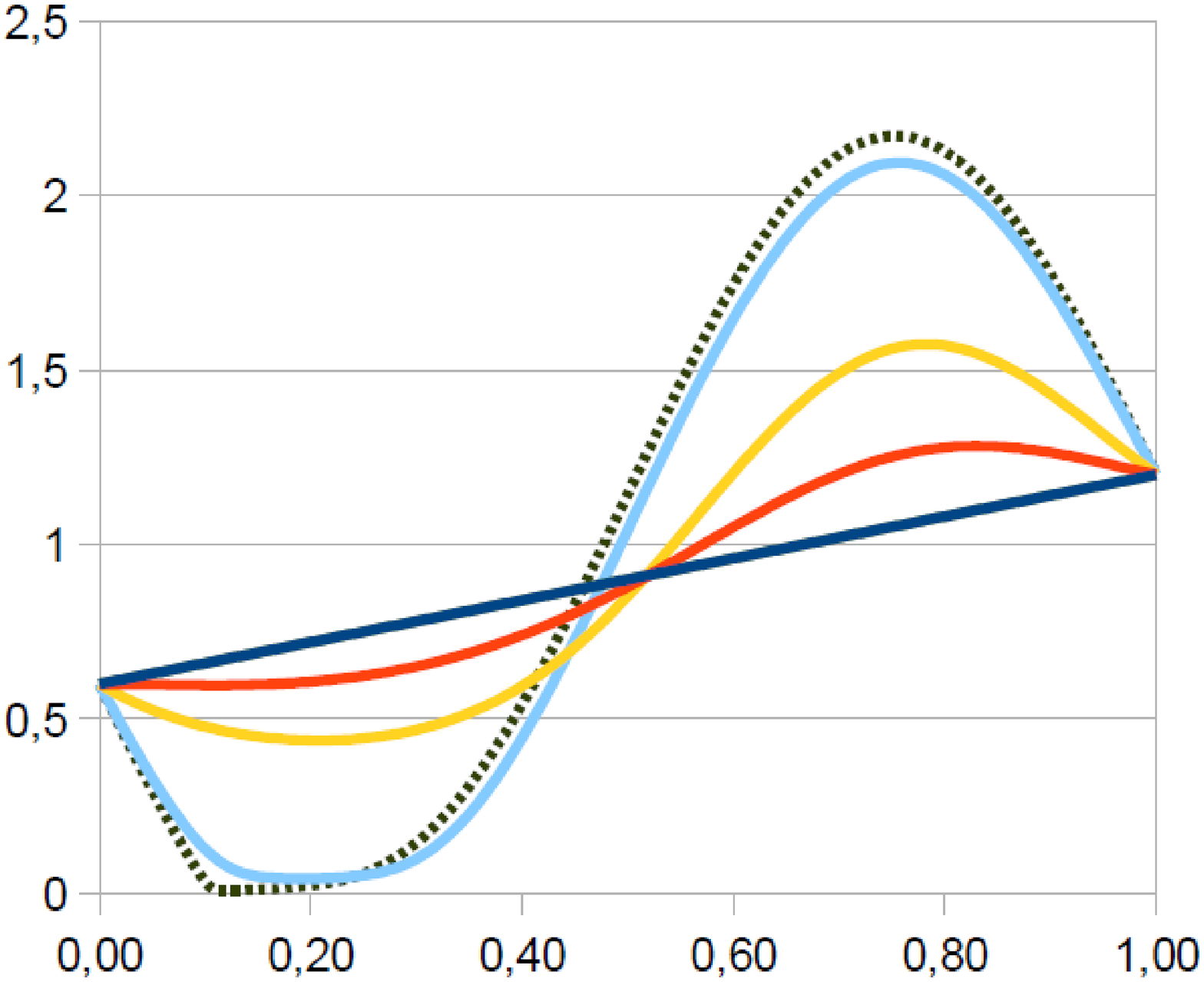}\\
\caption{Evolution towards the stationary state in the 1D case with sinusoidal potential (\ref{Vsin}). 
In each panel the initial datum is the straight line and the
asymptotic solution is the dotted black curve.
The other lines correspond to $t = 0.012$, $t =  0.024$, $t = 0.048$, $t = 0.1$. 
In panels (a) and (b) the two Dirichlet data are both supercritical 
(in particular, in panel (b) the asymptotic solution  is $u_0 - V(x_1)$). 
In panels (c) and (d), one of the data is subcritical (respectively, $u_1$ and $u_0$).
} 
\label{fig:0}
\end{center}
\end{figure}
\par
In panels (a), (b), (c) and (d) we have chosen, respectively,
$$
\begin{aligned}
 &(u_0,u_1) = (1.2,2.2),&  \quad   &(u_0,u_1) = (1.2,1.2),
 \\
 &(u_0,u_1) = (1.2,0.2),&  \quad   &(u_0,u_1) = (0.6,1.2).
\end{aligned}
$$
Then, in cases (a) and (b) both data are supercritical while, in case (c), $u_1$ is subcritical and, in case (d),
$u_0$ is subcritical.
We recall that Theorem \ref{Theo1D} guarantees the well-posedness of the stationary problem if at least one
of the Dirichlet data is supercritical.
\par
In case (b) since the potential (\ref{Vsin}) is compatible with equal Dirichlet data, the asymptotic solution is
exactly $u(x_1,\infty) = U(x_1) = u_0 - V(x_1)$ (corresponding to $c = 0$, see definition (\ref{Udef}) and Lemma \ref{Lemmac0}).
In cases (a) and (d), since $u_1 > u_0$, we have that $c>0$ and in case (c), where $u_1 < u_0$, we have that $c<0$
(see the proof of Theorem \ref{Theo1D} with $\Delta V = 0$). 
In the three cases (a), (c) and (d), therefore, the asymptotic solution is {\em not} $U(x_1)$ (which corresponds to $u_1 = u_0$, in the 
present case where $\Delta V = 0$).
Nevertheless, it is interesting to see that the asymptotic curves in Figure \ref{fig:0} tend to have the same character as $U(x_1)$, 
i.e.\ regular in the supercritical cases and piecewise regular in the subcritical cases
(however, as long as $t$ is finite, the solution is everywhere regular).
\subsection{Analysis of the current in the one-dimensional case: monotone potentials}
We consider again Equation (\ref{FDequation}) with the conditions (\ref{1Dcond}) and the following data:
and boundary conditions:
$$
u_\mathrm{in}=2-x_1,\qquad u_0=2, \qquad u_1=1.
$$
For this experiment, we have chosen the potential
$$
V(x_1) = -x_1.
$$
Since 
$$
u_0^\mathrm{crit}=0,\qquad u_1^\mathrm{crit}=1,
$$
the boundary condition $u_0$ in this simulation is supercritical.
\begin{figure}[h!]
\begin{center}
\includegraphics[width=10cm]{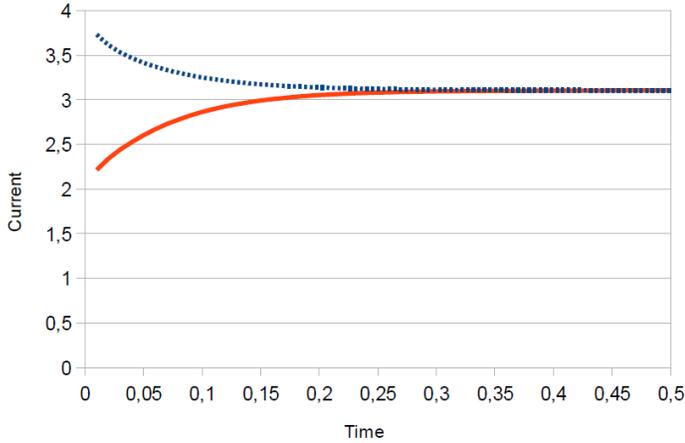}
\caption{Currents $J_L$ (blue dotted line) and $J_R$ (red continuous line) versus time under the action of the potential 
$V(x_1)=-x_1$.} \label{fig:2bisbis}
\end{center}
\end{figure}
In Figure \ref{fig:2bisbis}, we plot the current versus time that passes
through the boundaries $\Gamma_2$ (blue dotted line)
and $\Gamma_4 \cup \Gamma_5$ (red continuous line),
i.e.\ we visualize the time evolution of
$$
J_R :=\int_{\Gamma_2} J\cdot n_x\, dx_2 \textrm{ and }
J_L := \int_{\Gamma_4\cup \Gamma_5} J\cdot n_x\, dx_2
$$
respectively.

We note that, after a transient period, both currents at the extremities of the
device tend monotonically to the common value 3.11.

In the second numerical simulation, we have left unchanged the initial and boundary conditions,
but we have chosen a different potential, namely
$$
V(x_1) = -x_1+e^{-x_1^2}.
$$
This potential is non-linear, but is still monotone.
Here
$$
u_0^\mathrm{crit}=0,\qquad u_1^\mathrm{crit}=2-1/e \approx 1.63,
$$
hence the boundary condition $u_1$ is supercritical.
 
In Figure \ref{fig:2bis}, we plot the current versus time that passes
through the boundaries $\Gamma_2$ (blue dotted line)
and $\Gamma_4 \cup \Gamma_5$ (red continuous line)
respectively.

We note a different behaviour of $J_L$ and $J_R$ with respect to the previous simulation
in the transient period before reaching the equilibrium: $J_L$ is still a monotone
increasing function, but $J_R$ is no longer monotone. At time $t=1.4$ (not plotted in Figure \ref{fig:2bis}) the two currents at
the extremities of the domain have reached the common asymptotic value 4.01.

\begin{figure}[h!]
\begin{center}
\includegraphics[width=10cm]{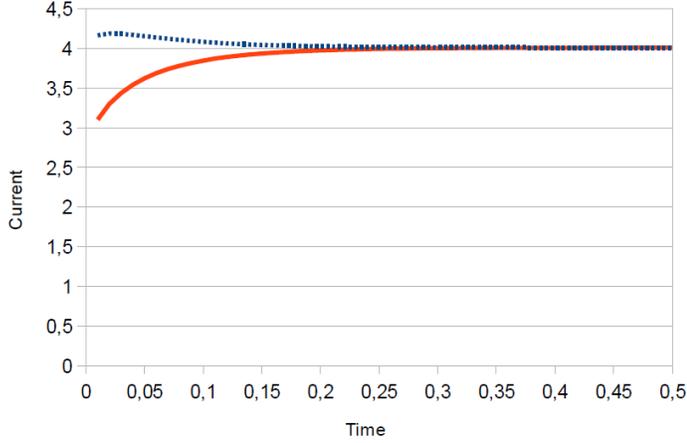}
\caption{Currents $J_L$ (blue dotted line) and $J_R$ (red continuous line) versus time under the action of the potential 
$V(x_1)= -x_1+e^{-x_1^2}$.} \label{fig:2bis}
\end{center}
\end{figure}

\subsection{Analysis of the current in the one-dimensional case: a potential barrier}

We now consider Equation (\ref{FDequation}) under the action of the potential
$$
V(x_1) = e^{-(x_1-.5)^2}.
$$
Here the critical values of the problem are
$$
u_0^\mathrm{crit}= u_1^\mathrm{crit}=1-e^{1/4} \approx 0.22.
$$

We compare the time evolution of the current at the extremities of the device
in two situations. 
The first one, whose results are plotted in Figure \ref{fig:1bis}, has been obtained with the initial data and the
boundary conditions
$$
u_\mathrm{in}=2-x_1,\qquad u_0 = 2, \qquad u_1 = 1.
$$

\begin{figure}[h!]
\begin{center}
\includegraphics[width=10cm]{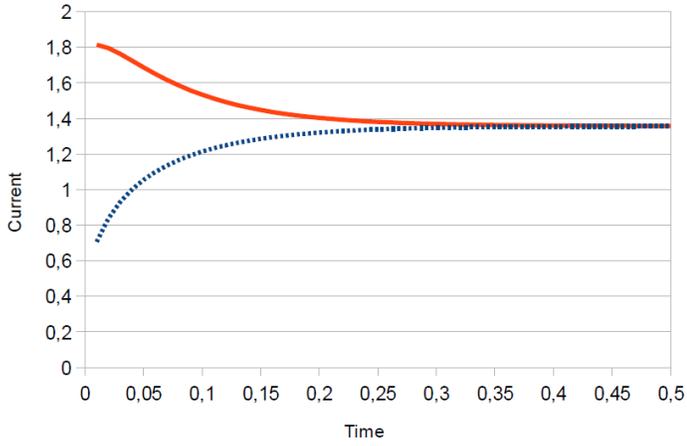}
\caption{Currents $J_L$ (blue dotted line) and $J_R$ (red continuous line) versus time under the action of the potential 
$V(x_1) = e^{-(x_1-.5)^2}$ and small gap between the boundary data in $\Gamma_2$ and
$\Gamma_4$.} \label{fig:1bis}
\end{center}
\end{figure}

The second situation, whose result are plotted in Figure \ref{fig:1bisbis}, has been obtained by imposing the following initial data and the boundary conditions:
$$
  u_\mathrm{in}=6-5x_1,\qquad u_0 = 6, \qquad u_1 = 1.
$$

\begin{figure}[h!]
\begin{center}
\includegraphics[width=10cm]{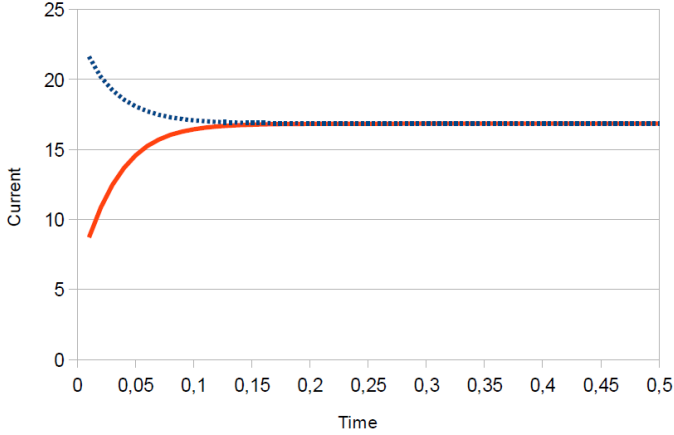}
\caption{Currents $J_L$ (blue dotted line) and $J_R$ (red continuous line) versus time under the action of the potential 
$V(x_1) = e^{-(x_1-.5)^2}$ and wide gap between the boundary data in $\Gamma_2$ and
$\Gamma_4$.} \label{fig:1bisbis}
\end{center}
\end{figure}

The numerical results of Figures \ref{fig:1bis} and \ref{fig:1bisbis} show that the stationary current flowing through the device 
in the case of a small density gap differ considerably from the one that is obtained in the case of a wide density gap 
(numerically we get a value which is close to 1.36 in the first case and the value 16.84 in the second). 
Moreover, in the first case, the convergence speed to the asymptotic state is slower than in the second case.

We have finally computed, for $a\in[0,3]$, the asymptotic common value at the
ends of the device versus $a$, which is a parameter that controls the gap
of the boundary data between the ends $\Gamma_2$ and $\Gamma_4$, through the choice
of the following initial and boundary conditions:
$$
u_\mathrm{in}=.5 +a(1-x),\qquad u_0 = .5 +a, \qquad u_1 = .5.
$$
Both boundary data are supercritical.
\begin{figure}[h!]
\begin{center}
\includegraphics[width=10cm]{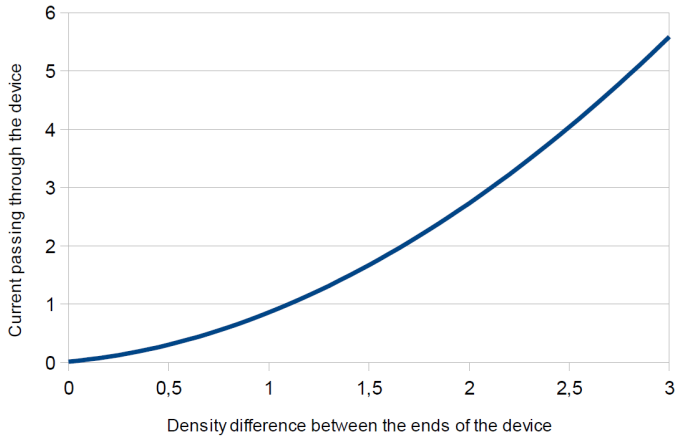}
\caption{Asymptotic current at the ends of the device under the action 
of the potential $V(x_1) = e^{-(x_1-.5)^2}$ versus the gap between the boundary data in $\Gamma_2$ and
$\Gamma_4$.} \label{fig:6}
\end{center}
\end{figure}

The results, plotted in Figure \ref{fig:6}, show that the current is zero when $a=0$ and that it is strictly monotone. The profile is parabolic with a good approximation. This
simulation has been very time-consuming, since we had to repeat, for each
element of the discretization of $a\in[0,3]$, a computation in long-time that gives the asymptotic common value of
the currents at the ends of the device.
The discretization between the values of $a$ is $\Delta a=0.05$, and the
final time of each simulation has been $t=2.5$. The discrepancies observed between the numerical values of $J_R(t=2.5)$ and
$J_L(t=2.5)$ are, for each value of $a$, beyond the resolution of Figure \ref{fig:6}.

\subsection{Analysis of the current in a two-dimensional device in a non-symmetric situation}

We conclude the description of our simulations with a test-case that is widely used
in the literature \cite{MRSbook}.
Consider the system defined by 
(\ref{FDequation}) under the action of the linear potential 
$$
 V(x) = 1-x_1.
$$
In this case, the initial data and boundary conditions are
$$
u_\mathrm{in}(x)=\cos(\pi x_1)+2,
$$
and 
$$
J|_{\Gamma_1}=J|_{\Gamma_3}=J|_{\Gamma_4}=0, \qquad
u|_{\Gamma_2}=1,\qquad 
u|_{\Gamma_5}=3.
$$

Note that, even though $u_\mathrm{in}$  and $V$ depend only on the first variable $x_1$, 
the problem is genuinely two-dimensional because of the boundary conditions.
\par
\begin{figure}[h!]
\begin{center}
a)\ \includegraphics[height=3.8cm]{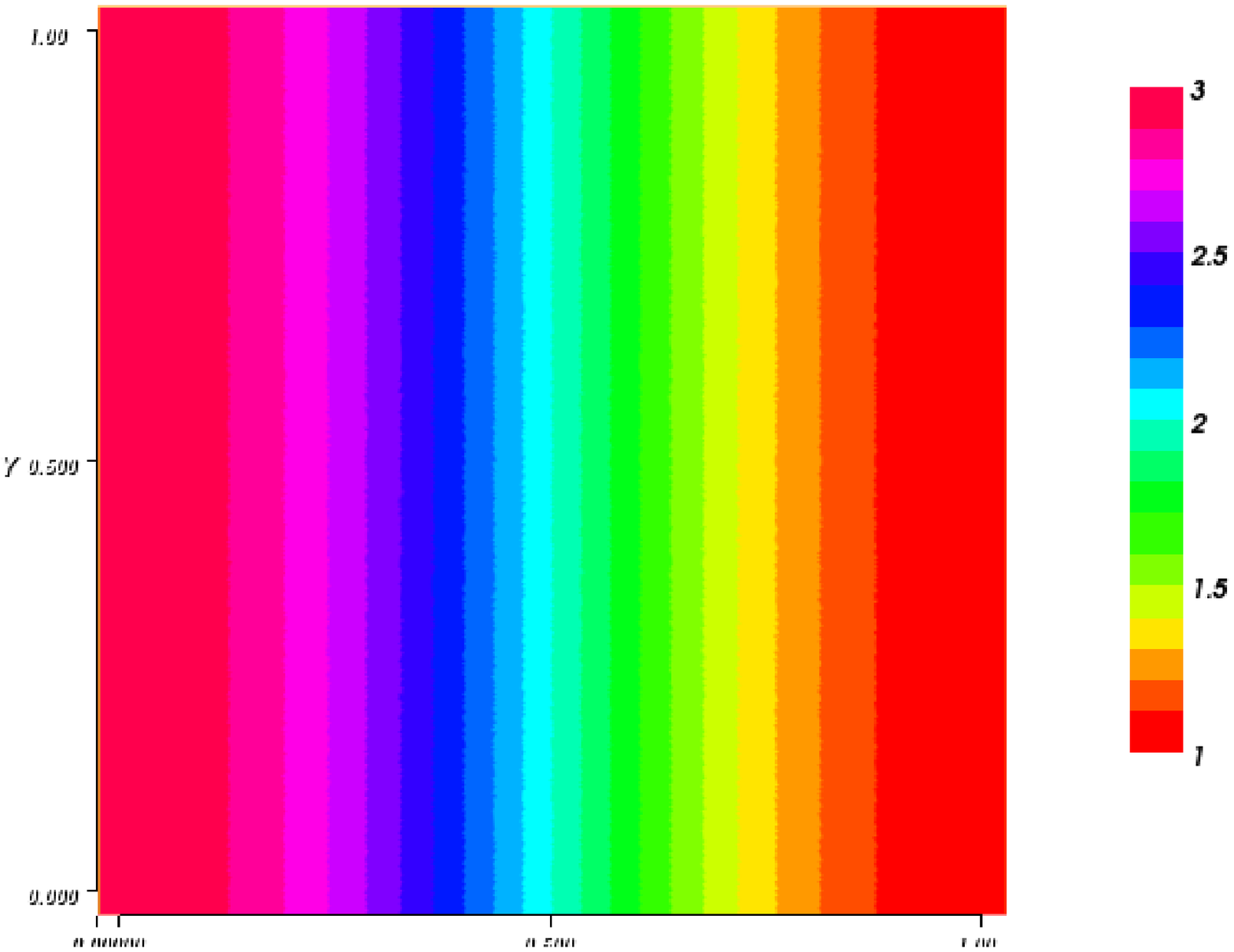}\
b)\ \includegraphics[height=3.8cm]{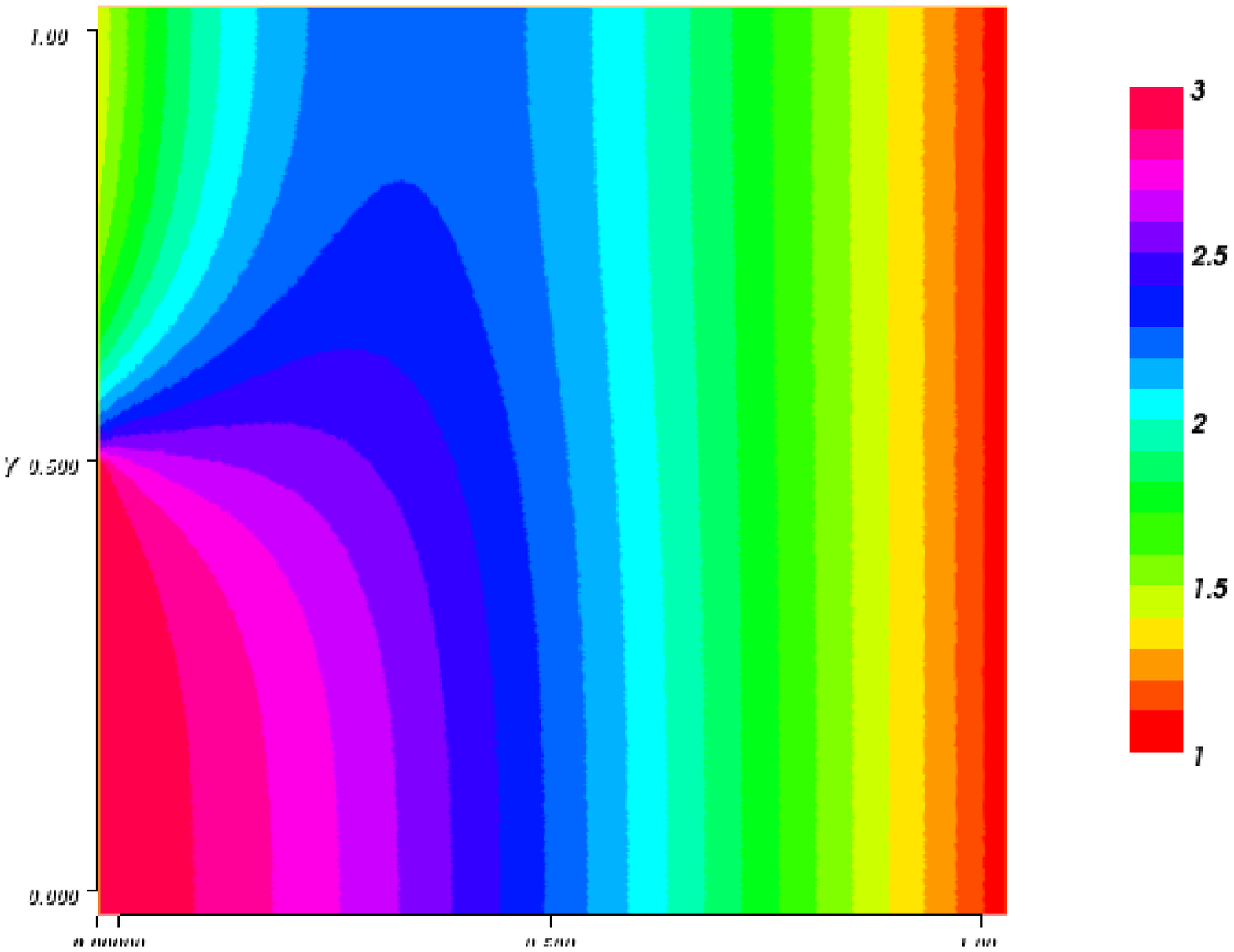}\\[1.2ex]
c)\ \includegraphics[height=3.8cm]{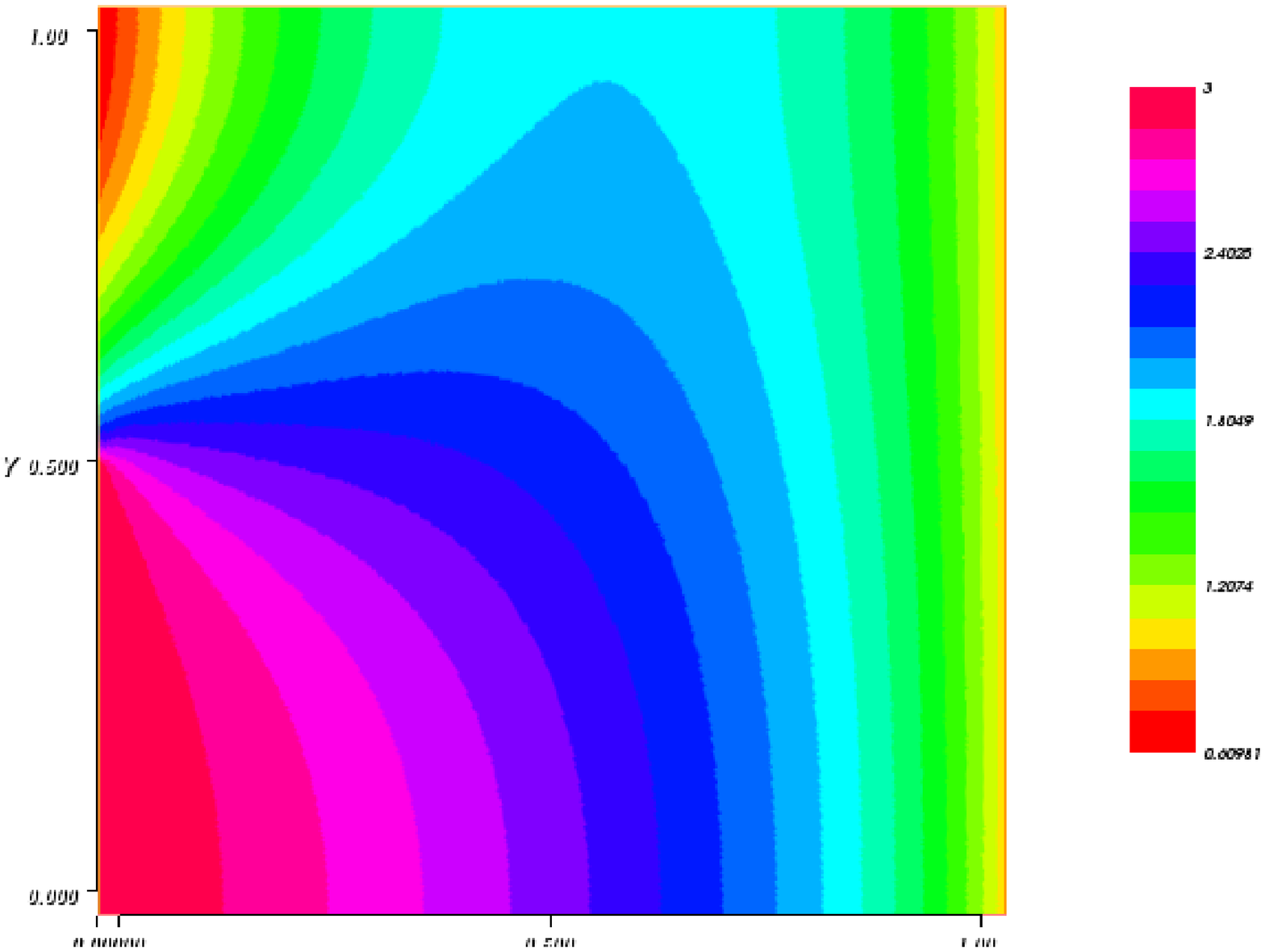}\
d)\ \includegraphics[height=3.8cm]{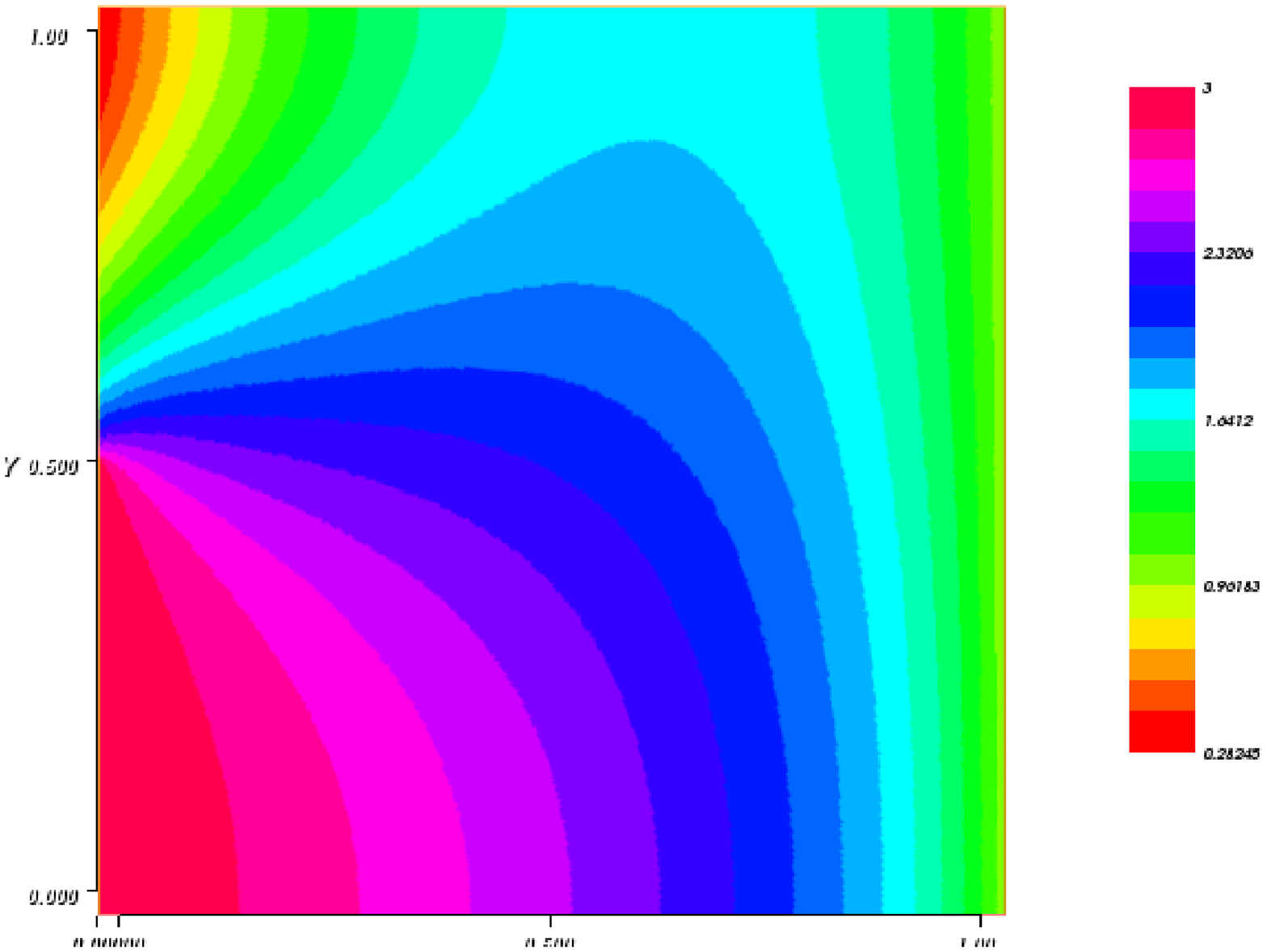}\\ [1.2ex]
e)\ \includegraphics[height=3.8cm]{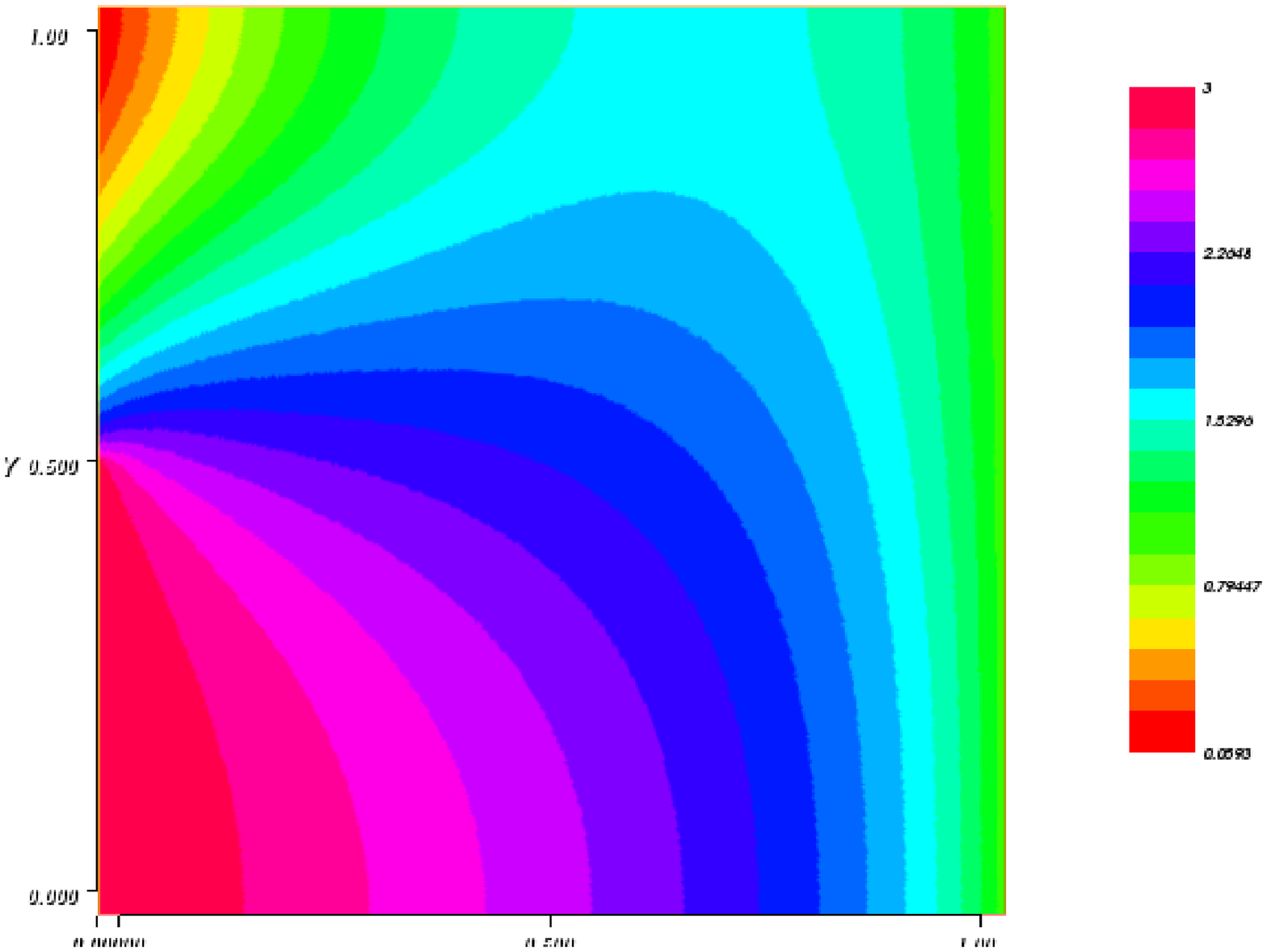}
f)\ \includegraphics[height=3.8cm]{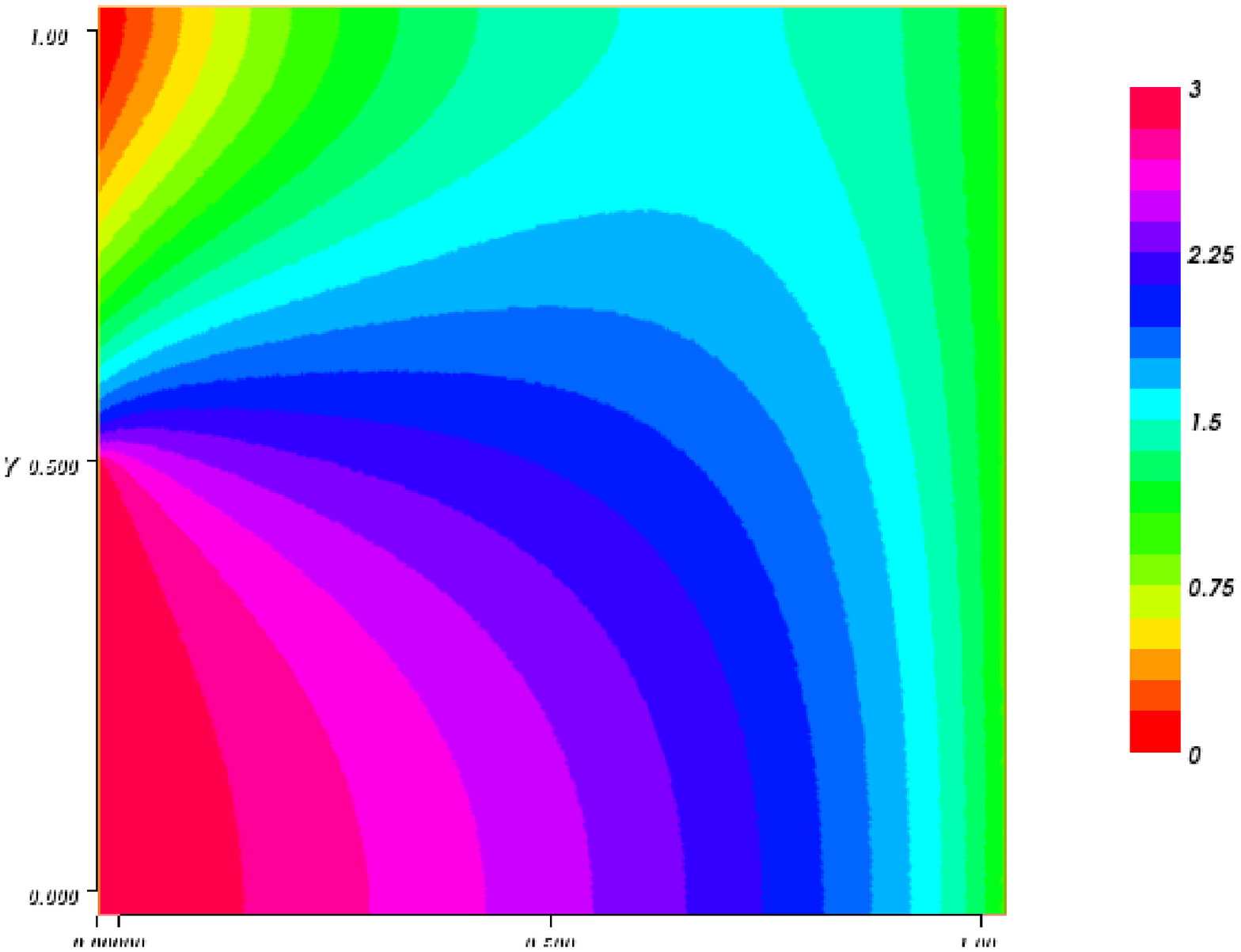}
\caption{Contours of the density $u$ at six different instants: a) $t=0$ (initial condition), 
b) $t=0.10$, c) $t=à0.30$,
d) $t=0.50$, e) $t=0.80$ and f) $t=1.00$.} \label{fig:4}
\end{center}
\end{figure}

In Figure \ref{fig:4}, we show the time evolution of the unknown density $u$. 
At time $t=1$ the system has already reached an almost stationary configuration.

\begin{figure}[h!]
\begin{center}
\includegraphics[width=10cm]{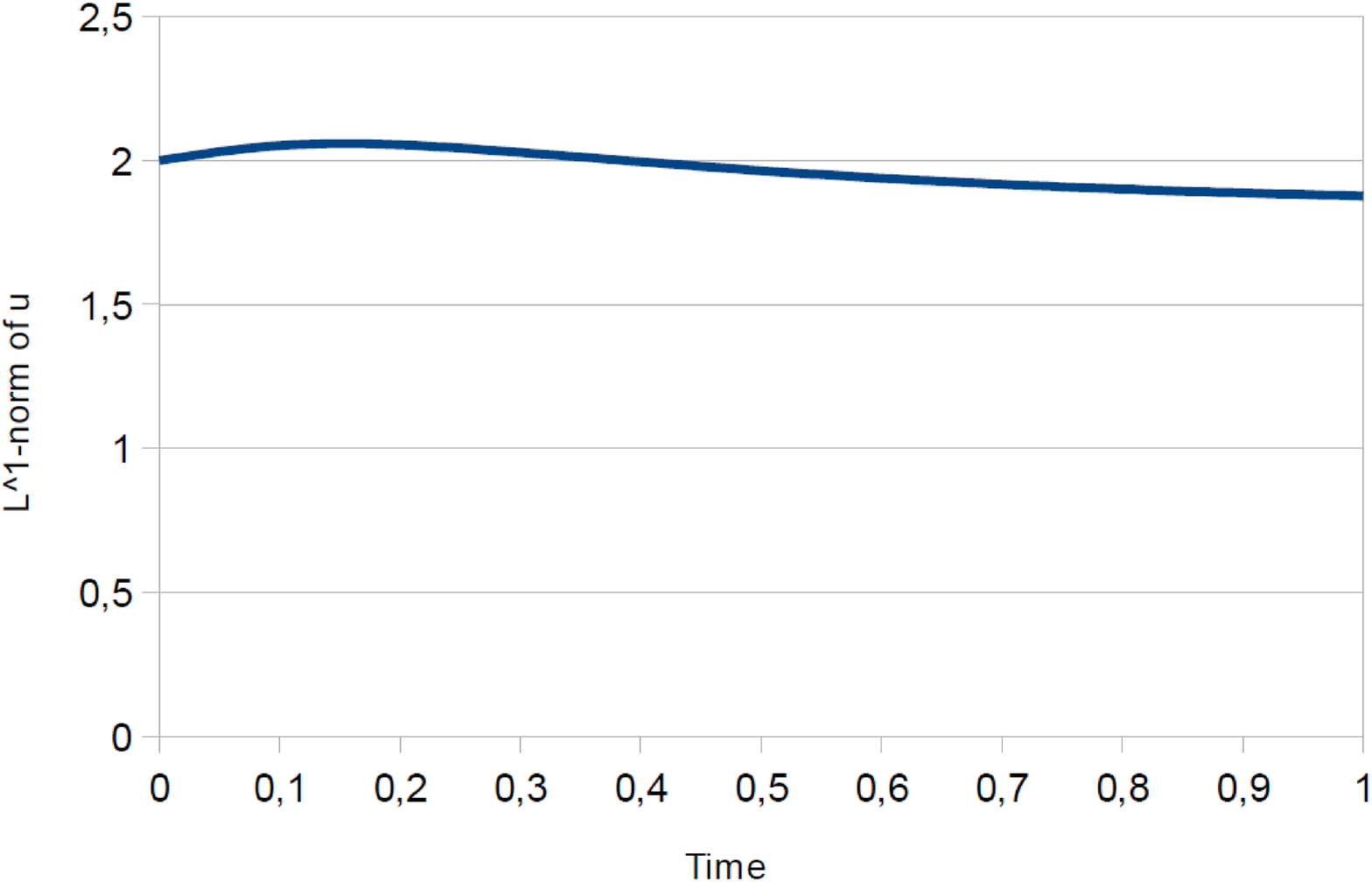}
\caption{Time evolution of $\Vert u\Vert_1$.} \label{fig:3}
\end{center}
\end{figure}

We also show in Figure \ref{fig:3} the time evolution of the $L^1$-norm of the density $u$. 
Note that the total mass of the problem is not
conserved in time, but has a non-monotone behaviour around the value $\Vert u_\mathrm{in}\Vert_{L^1(\Omega)}=2$.



\begin{thebibliography}{10}

\bibitem{AL}
{\sc H.~W. Alt, and S. Luckhaus},
{\it Quasilinear elliptic-parabolic differential equations},
Math.\ Z., 183 (1983), pp.\ 311--341.

\bibitem{bar-gol-lev-89}
{\sc C. Bardos, F. Golse, and C.~D. Levermore},
{\it  Sur les limites asymptotiques de la th\'eorie cin\'etique conduisant \`a la dynamique des fluides incompressibles},
C.\ R.\ Acad.\ Sci.\ Paris S{\'e}r.\ I Math., 309 (1989), pp.\ 727--732.

\bibitem{bar-gol-lev-91}
{\sc C. Bardos, F. Golse, and C.~D. Levermore},
{\it  Fluid dynamic limits of kinetic equations. I. Formal derivations},
J.\ Stat.\ Phys., 63 (1991), pp.\ 323--344.

\bibitem{bar-gol-lev-93}
{\sc C. Bardos, F. Golse, and C.~D. Levermore},
{\it  Fluid dynamic limits of kinetic equations. II. Convergence proofs for the Boltzmann equation},
Comm.\ Pure Appl.\ Math., 46 (1993), pp.\  667--753.

\bibitem{BarFroMor14}
{\sc L. Barletti, G. Frosali, and O. Morandi},
{\it  Kinetic and hydrodynamic models for multi-band quantum transport in crystals},
in Multi-Band Effective Mass Approximations: Advanced Mathematical Models and Numerical Techniques,
M. Ehrhardt and T.\ Koprucki, eds.,
Springer-Verlag, Berlin, 2014. Chap.\ 1, pp.\  3--56.

\bibitem{BaCi}
{\sc L. Barletti, and C.\ Cintolesi},
{\it  Derivation of isothermal quantum fluid equations with Fermi-Dirac and Bose-Einstein statistics.}
J.\ Stat.\ Phys., 148 (2012), pp.\ 353--386.

\bibitem{MR1866628}
{\sc P. Biler, J. Dolbeault, P.A. Markowich},
{\it Large time asymptotics of nonlinear drift-diffusion systems with Poisson coupling}. 
The Sixteenth International Conference on Transport Theory, Part II.
Transport Theory Statist. Phys. 30 (2001), pp.\  521?536. 

\bibitem{bou-gre-pav-sal-13}
{\sc L. Boudin, B. Grec, M. Pavi\'c, and F. Salvarani},
{\it Diffusion asymptotics of a kinetic model for gaseous mixtures},
Kinet.\ Relat.\ Models, 6 (2013), pp.\ 137--157.

\bibitem{bou-gre-sal-13}
{\sc L. Boudin, B. Grec, and F. Salvarani},
{\it The Maxwell-Stefan diffusion limit for a kinetic model of mixtures},
Acta Appl.\ Math.\  (in press).

\bibitem{CJMTU}
{\sc J.~A. Carrillo, A. J\"ungel, P.~A. Markowich, G. Toscani, and A. Unterreiter},
{\it  Entropy dissipation methods for degenerate parabolic problems and generalized Sobolev inequalities},
Monatsh.\ Math., 133 (2001), pp.\ 1--82.

\bibitem{Cerci88}
{\sc C. Cercignani},
{\it The Boltzmann equation and its applications}, 
Springer-Verlag, New York, 1988.

\bibitem{des-mon-sal}
{\sc L. Desvillettes, R. Monaco, and F. Salvarani},
{\it  A kinetic model allowing to obtain the energy law of polytropic gases in the presence of chemical reactions},
 Eur.\ J.\ Mech.\ B Fluids, 24 (2005), pp.\ :219--236.
 
 \bibitem{MR1835610}
{\sc J. I. D{\'{\i}}az, G. Galiano, A. J{\"u}ngel},
{\it On a quasilinear degenerate system arising in semiconductors theory. {I}. {E}xistence and uniqueness of solutions},
 Nonlinear Anal. Real World Appl. 2 (2001), pp.\ 305--336.


\bibitem{gol-stray-04}
{\sc F. Golse, and L. Saint-Raymond},
{\it  The Navier-Stokes limit of the Boltzmann equation for bounded collision kernels},
Invent.\ Math., 155 (2004), pp.\ 81--161.

\bibitem{gol-stray-09}
{\sc F. Golse, and L.Saint-Raymond},
{\it  The incompressible Navier-Stokes limit of the Boltzmann equation for hard cutoff potentials},
J.\ Math.\ Pures Appl.\ (9), 91 (2009), pp.\ 508--552.


\bibitem{MR1818867}
{\sc A. J\"ungel},
{\it Quasi-hydrodynamic semiconductor equations},
Birkh\"auser Verlag, Basel, 2001.

\bibitem{Jungel}
{\sc A. J\"ungel},
{\it Transport Equations for Semiconductors},
Springer-Verlag, Berlin, 2009.

\bibitem{MR1479577}
{\sc A. J\"ungel, P. Pietra},
{\it A discretization scheme for a quasi-hydrodynamic semiconductor model},
Math. Models Methods Appl. Sci. 7 (1997), pp.\ 935--955.

\bibitem{JKP11}
{\sc A. J{\"u}ngel, S. Krause, and P. Pietra},
{\it Diffusive semiconductor moment equations using Fermi-Dirac statistics},
Z.\ Angew.\ Math.\ Phys., 62 (2011), pp.\ 623--639.

\bibitem{Madelung}
{\sc E. Madelung},
{\it Quantentheorie in hydrodynamischer form},
Z.\ Phys., 40 (1926) pp.\ 322--326.

\bibitem{MRSbook}
{\sc P.~A. Markowich, C.~A. Ringhofer, and C. Schmeiser.
{\it Semiconductor equations},
Springer-Verlag, Vienna, 1990.

\bibitem{NIST}
{\sc F.~ W.~J. Olver, D.~W. Lozier, R.~F. Boisvert, and C.~W. Clark},
{\it NIST Handbook of Mathematical Functions},
Cambridge University Press, Cambridge, 2010.

\bibitem{TR10}
{\sc M. Trovato and L. Reggiani},
{\it Quantum hydrodynamic models from a maximum entropy principle},
J.\ Phys.\ A}, 43 (2010), article 102001.

\bibitem{Vazquez07}
{\sc J.~L. V\'azquez},
{\it The Porous Medium Equation: Mathematical Theory},
Oxford University Press, Oxford, 2007.

\end{thebibliography}
\end{document}